\documentclass{amsart} 

\usepackage{amsmath, amssymb, amsfonts, tikz, hyperref, comment, url, todonotes, array}
\usetikzlibrary{math}
\usetikzlibrary{decorations.pathreplacing}
\usetikzlibrary{arrows.meta}

\newtheorem{theorem}{Theorem}[section]
\newtheorem{lemma}[theorem]{Lemma}
\newtheorem{proposition}[theorem]{Proposition}

\theoremstyle{definition}
\newtheorem{example}[theorem]{Example}
\newtheorem{remark}[theorem]{Remark}
\newtheorem{definition}[theorem]{Definition}
\newtheorem*{definition*}{Definition}
\newtheorem*{theorem*}{Theorem}
\newtheorem*{proposition*}{Proposition}
\newtheorem*{corollary*}{Corollary}

\newcommand{\R}{\mathbb{R}}
\newcommand{\GL}{\mathrm{GL}}
\renewcommand{\S}{\mathfrak{S}}
\newcommand{\Supp}{\mathrm{Supp}}

\author{Yibo Gao}
\address{Department of Mathematics, Massachusetts Institute of Technology, \mbox{Cambridge, MA 02139}}
\email{\href{mailto:gaoyibo@mit.edu}{{\tt gaoyibo@mit.edu}}}

\author{Kaarel H\" anni}
\address{Department of Mathematics, Massachusetts Institute of Technology, \mbox{Cambridge, MA 02139}}
\email{\href{mailto:kaarelh@mit.edu}{{\tt kaarelh@mit.edu}}}

\begin{document}
\title{Boolean elements in the Bruhat order}
\date{\today}

\begin{abstract}
We show that $w\in W$ is boolean if and only if it avoids a set of Billey-Postnikov patterns, which we describe explicitly. Our proof is based on an analysis of inversion sets, and it is in large part type-uniform. We also introduce the notion of linear pattern avoidance, and show that boolean elements are characterized by avoiding just the $3$ linear patterns $s_1 s_2 s_1 \in W(A_2)$, $s_2 s_1 s_3 s_2 \in W(A_3)$, and $s_2 s_1 s_3 s_4 s_2 \in W(D_4)$. 

We also consider the more general case of $k$-boolean Weyl group elements. We say that $w\in W$ is $k$-boolean if every reduced expression for $w$ contains at most $k$ copies of each generator. We show that the $2$-boolean elements of the symmetric group $S_n$ are characterized by avoiding the patterns $3421,4312,4321,$ and $456123$, and give a rational generating function for the number of $2$-boolean elements of $S_n$.

\end{abstract}
\maketitle

\section{Introduction}\label{sec:intro}
Billey and Postnikov \cite{billey2005smoothness} defined a notion of pattern avoidance in Weyl groups, which efficiently characterizes those Weyl group elements $w$ whose corresponding Schubert variety $X_w$ is (rationally) smooth, for arbitrary Weyl groups, generalizing the well-known result of Lakshmibai and Sandhya \cite{Lakshmibai-Sandhya} that says for a permutation $w$, its Schubert variety $X_w$ is smooth if and only if $w$ avoids 3412 and 4231. Since then, Billey-Postnikov patterns (BP patterns), besides geometric importance, have seen many combinatorial applications as well, characterizing fully commutative elements \cite{billey2005smoothness}, chromobruhatic elements \cite{woo2018hultman}, separable elements \cite{gaetz2019splittings,gaetz2019separable}, and so on.

In this paper, we showcase another combinatorial application of BP patterns (Definition~\ref{def:BP}), by characterizing \textit{boolean} elements of arbitrary Weyl groups, generalizing a result by Tenner \cite{tenner2007pattern} for the symmetric group, who showed that a permutation $w$ is boolean if and only if it avoids 321 and 3412. Let $\Phi$ be any finite crystallographic root system with Weyl group $W=W(\Phi)$ (see more background in Section~\ref{sec:background}).
\begin{definition}
An element $w\in W$ is called \textit{boolean} if the interval $[\mathrm{id},w]$ in the (strong) Bruhat order is isomorphic to a Boolean lattice.
\end{definition}

Here is the first version of our main theorem.
\begin{theorem}\label{thm:main}
Let $\Phi$ be a root system. An element $w\in W(\Phi)$ is boolean if and only if $w$ avoids all the BP patterns in Table \ref{tab:booleanpatterns}.

\begin{table}[h!]
\centering
\begin{tabular}{c|c|c}
type     & forbidden patterns & \# patterns \\\hline
$A_2$    & $s_1s_2s_1=s_2s_1s_2$ (321) & 1\\\hline
$A_3$ & $s_2s_1s_3s_2$ (3412) & 1\\\hline
$B_2=C_2$ & $s_1s_2s_1$, $s_2s_1s_2$, $s_1s_2s_1s_2=s_2s_1s_2s_1$ & 3\\\hline
$B_3$ & $s_2 s_1 s_3 s_2$ & 1\\\hline
$C_3$ & $s_2 s_1 s_3 s_2$ & 1\\\hline
$D_4$ & $s_2s_1s_3s_4s_2$ & 1\\\hline
$G_2$ & all patterns of Coxeter length at least 3 & 7 \\\hline
\end{tabular}
\caption{Forbidden patterns for boolean elements in Weyl groups}
\label{tab:booleanpatterns}
\end{table}
\end{theorem}
See Table~\ref{tab:labelDynkin} for labels on the Dynkin diagram, where we use $s_i$ to denote the reflection across the simple root $\alpha_i$. We omit root systems of rank 2 since no confusion will arise.

Theorem~\ref{thm:main} is notable because in \cite{tenner2007pattern}, Tenner showed that an element being boolean is equivalent to avoiding 10 patterns in type $B$ and avoiding 20 patterns in type $D$, with a certain notion of pattern avoidance for signed permutations, while we only need 7 BP patterns in type $B$ and 3 BP patterns in type $D$.

\begin{table}[h!]
\centering
\begin{tabular}{ m{0.66cm}| m{2.7cm}| m{1.6cm}| m{5cm} }
type & Dynkin diagram & pattern $\pi$ & inversions $I_{\Phi}(\pi)$ \\\hline
$A_2$ & 
\begin{tikzpicture}
\node at (0,0) {$\bullet$};
\node[above] at (0,0) {$\alpha_1$};
\node at (1,0) {$\bullet$};
\node[above] at (1,0) {$\alpha_2$};
\draw(0,0)--(1,0);
\end{tikzpicture}
& $s_1s_2s_1$ & $\{\alpha_1,\alpha_2,\alpha_1{+}\alpha_2\}$\\\hline
$A_3$ & 
\begin{tikzpicture}
\node at (0,0) {$\bullet$};
\node[above] at (0,0) {$\alpha_1$};
\node at (1,0) {$\bullet$};
\node[above] at (1,0) {$\alpha_2$};
\node at (2,0) {$\bullet$};
\node[above] at (2,0) {$\alpha_3$};
\draw(0,0)--(2,0);
\end{tikzpicture}
& $s_2s_1s_3s_2$ & $\{\alpha_2,\alpha_1{+}\alpha_2,\alpha_2{+}\alpha_3,\alpha_1{+}\alpha_2{+}\alpha_3\}$ \\\hline
$B_3$ &
\begin{tikzpicture}
\node at (0,0) {$\bullet$};
\node[above] at (0,0) {$\alpha_1$};
\node at (1,0) {$\bullet$};
\node[above] at (1,0) {$\alpha_2$};
\node at (2,0) {$\bullet$};
\node[above] at (2,0) {$\alpha_3$};
\draw(0,0)--(1,0);
\draw(1,0.03)--(2,0.03);
\draw(1,-0.03)--(2,-0.03);
\draw(1.4,0.1)--(1.6,0)--(1.4,-0.1);
\end{tikzpicture}
& $s_2s_1s_3s_2$ & $\{\alpha_2,\alpha_1+\alpha_2,\alpha_2+\alpha_3,\alpha_1+2\alpha_2+2\alpha_3\}$ \\\hline
$C_3$ &
\begin{tikzpicture}
\node at (0,0) {$\bullet$};
\node[above] at (0,0) {$\alpha_1$};
\node at (1,0) {$\bullet$};
\node[above] at (1,0) {$\alpha_2$};
\node at (2,0) {$\bullet$};
\node[above] at (2,0) {$\alpha_3$};
\draw(0,0)--(1,0);
\draw(1,0.03)--(2,0.03);
\draw(1,-0.03)--(2,-0.03);
\draw(1.6,0.1)--(1.4,0)--(1.6,-0.1);
\end{tikzpicture}
& $s_2s_1s_3s_2$ & $\{\alpha_2,\alpha_1+\alpha_2,2\alpha_2+\alpha_3,\alpha_1+2\alpha_2+\alpha_3\}$ \\\hline
$D_4$ & 
\begin{tikzpicture}
\node at (0,0) {$\bullet$};
\node[above] at (0,0) {$\alpha_1$};
\node at (1,0) {$\bullet$};
\node[above] at (1,0) {$\alpha_2$};
\node at (1.9,0.4) {$\bullet$};
\node[right] at (1.9,0.4) {$\alpha_3$};
\node at (1.9,-0.4) {$\bullet$};
\node[right] at (1.9,-0.4) {$\alpha_4$};
\draw(0,0)--(1,0)--(1.9,0.4);
\draw(1,0)--(1.9,-0.4);
\end{tikzpicture}
& $s_2s_1s_3s_4s_2$ & $\{\alpha_2,\alpha_1{+}\alpha_2,\alpha_2{+}\alpha_3,\alpha_2{+}\alpha_4,$ $ \alpha_1{+}2\alpha_2{+}\alpha_3{+}\alpha_4,\}$ \\\hline
\end{tabular}
\caption{Patterns of interest and their inversions}
\label{tab:labelDynkin}
\end{table}

Moreover, we also introduce a new notion of \textit{linear patterns} (Definition~\ref{def:linear}), which simultaneously generalizes the classical folding of root systems and root system embedding \cite{billey2005smoothness}. This notion allows us to derive an even simpler characterization of boolean elements, which requires only the same $3$ patterns in all types. The following is the second version of our main theorem.
\begin{theorem}\label{thm:linear}
Let $\Phi$ be an irreducible root system. An element $w\in W(\Phi)$ is boolean if and only if $w$ avoids the linear patterns $s_1 s_2 s_1 \in W(A_2)$, $s_2 s_1 s_3 s_2 \in W(A_3)$, and $s_2 s_1 s_3 s_4 s_2 \in W(D_4)$.
\end{theorem}

The paper is organized as follows. In Section~\ref{sec:background}, we provide necessary background and definitions on Weyl groups and pattern avoidance. In Section~\ref{sec:mainpf}, we prove the two versions of our main theorems by first proving Theorem~\ref{thm:linear} and then deriving Theorem~\ref{thm:main} from Theorem~\ref{thm:linear}. Our proof is largely type-uniform and is completely independent of that of Tenner \cite{tenner2007pattern,tenner2020interval}, even in the case of type $A$ root systems whose Weyl group is isomorphic to the symmetric group. Finally in Section~\ref{sec:k-boolean}, we go back to the symmetric group and generalize the notion of boolean permutations to $k$-boolean permutations, characterize $2$-boolean permutations by pattern avoidance (as the case $k\geq3$ does not seem to be governed by pattern avoidance), and enumerate them.

\section{Background on Weyl groups and patterns}\label{sec:background}
We refer readers to \cite{humphreys1978introduction} for a detailed treatment on root systems.

Throughout the paper, let $\Phi\subset E$ be a finite crystallographic root system of rank $r$ inside an Euclidean space $E\simeq\R^r$ with a positive definite symmetric bilinear form $\langle-,-\rangle$. We fix a choice of positive roots $\Phi^+\subset\Phi$ which corresponds to a set of simple roots $\Delta=\{\alpha_1,\ldots,\alpha_r\}$. Let $W=W(\Phi)$ be its Weyl group, which is a finite subgroup of $\GL(E)$ generated by reflections $s_{\alpha}\in\GL(E)$ for all roots $\alpha$, or equivalent, by $s_{\alpha}$'s for $\alpha\in\Delta$. For simplicity of notations, we write $s_i$ for $s_{\alpha_i}$ where $\alpha_i\in\Delta$ and we call these reflections \textit{simple reflections}. 

The \textit{(strong) Bruhat order} on $W$, which naturally comes from the Bruhat decomposition of the flag variety, is defined to be the transitive closure of $w\lessdot ws_{\beta}$ if $\ell(w)=\ell(ws_{\beta})-1$, where $\ell$ denotes the Coxeter length. There is a minimum $\mathrm{id}$ and a maximum $w_0$ of the Bruhat order. The Bruhat order satisfies the \textit{subword property}, that says if $v<u\in W$ and $u=s_{i_1}\cdots s_{i_{\ell}}$ is a reduced expression, then there exists a subword of $s_{i_1}\cdots s_{i_{\ell}}$ that is a reduced expression for $v$.

A root system $\Phi$ is \textit{irreducible} if it cannot be properly partitioned into $\Phi_1\sqcup\Phi_2$ such that $\langle \beta_1,\beta_2\rangle=0$ for all $\beta_1\in\Phi_1$ and $\beta_2\in\Phi_2$. Irreducible root systems can be completely classified into 4 infinite families $A_n,B_n,C_n,D_n$ and exceptional types $E_6,E_7,E_8,F_4,G_2$. We adopt the following conventions for the classical types, as in \cite{humphreys1978introduction}:
\begin{itemize}
\item type $A_{n-1}$: $\Phi=\{e_i-e_j\:|\:1\leq i,j\leq n\}\subset\R^n/(1,\ldots,1)$, $\Phi^+=\{e_i-e_j\:|\:1\leq i<j\leq n\}$, $\Delta=\{e_i-e_{i+1}\:|\:1\leq i\leq n-1\}$;
\item type $B_n$: $\Phi=\{\pm e_i\pm e_j\:|\: 1\leq i<j\leq n\}\cup\{\pm e_i\:|\:1\leq i\leq n\}$, $\Phi^+=\{e_i\pm e_j\:|\:1\leq i<j\leq n\}\cup\{e_i\:|\:1\leq i\leq n\}$, $\Delta=\{e_i-e_{i+1}\:|\:1\leq i\leq n-1\}\cup\{e_n\}$;
\item type $C_n$: $\Phi=\{\pm e_i\pm e_j\:|\: 1\leq i<j\leq n\}\cup\{\pm 2e_i\:|\:1\leq i\leq n\}$, $\Phi^+=\{e_i\pm e_j\:|\:1\leq i<j\leq n\}\cup\{2e_i\:|\:1\leq i\leq n\}$, $\Delta=\{e_i-e_{i+1}\:|\:1\leq i\leq n-1\}\cup\{2e_n\}$;
\item type $D_n$: $\Phi=\{\pm e_i\pm e_j\:|\: 1\leq i<j\leq n\}$, $\Phi^+=\{e_i\pm e_j\:|\:1\leq i<j\leq n\}$, $\Delta=\{e_i-e_{i+1}\:|\:1\leq i\leq n-1\}\cup\{e_{n-1}+e_n\}$.
\end{itemize}
Note that the root system of type $B_2$ is isomorphic to $C_2$. And when we talk about root system of type $D_n$, we assume $n\geq4$ as $D_3$ is the same as $A_3$.

The \textit{root poset} is the partial order on $\Phi^+$ such that $\alpha\leq\beta\in\Phi^+$ if $\beta-\alpha$ can be written as a nonnegative (integral) linear combination of simple roots. The minimal elements of the root poset are precisely the simple roots $\Delta$ and there exists a unique maximum of the root poset called the \textit{highest root}. The root poset can be given the structure of a graded poset with the rank of a root being the sum of coefficients of this root in the simple root basis, known as the \emph{height} of this root. We say that a positive root $\beta$ is \textit{supported on} a simple root $\alpha\in\Delta$ if $\beta\geq\alpha$ in the root poset. Define the \textit{support} of $\beta$ to be $$\Supp(\beta):=\{\alpha\in\Delta\:|\:\beta\text{ is supported on }\alpha\}\subset\Delta.$$

For $w\in W(\Phi)$, its \textit{inversion set} is
$$I_{\Phi}(w)=\{\beta\in\Phi^+\:|\:w\beta\in\Phi^-\}.$$
We say that $\beta$ is an \textit{inversion} of $w$ if $\beta\in I_{\Phi}(w)$, and a \textit{(right) descent} of $w$ if $\beta\in I_{\Phi}(w)\cap\Delta$ is an inversion and also a simple root.
It is a standard fact that $\ell(w)=|I_{\Phi}(w)|$. The following lemma follows from definitions, with proof omitted.
\begin{lemma}\label{lem:inv-after-si}
Let $w\in W(\Phi)$ and $\alpha\in\Delta$ such that $\ell(ws_{\alpha})=\ell(w)+1$. Then
$$I_{\Phi}(ws_{\alpha})=s_{\alpha}I_{\Phi}(w)\cup\{\alpha\}.$$
\end{lemma}

The next proposition is useful and well-known (see for example \cite{hohlweg2016On}).
\begin{proposition}\label{prop:biconvex}
The inversion set uniquely characterizes a Weyl group element. In other words, $I_{\Phi}:W\rightarrow 2^{\Phi^+}$ is injective. Moreover, a subset $I\subset\Phi^+$ is the inversion set of some Weyl group elements if and only if it is \textit{biconvex}; that is, if and only if:
\begin{enumerate}
\item if $\alpha,\beta\in I$, $\alpha+\beta\in\Phi^+$, then $\alpha+\beta\in I$ and,
\item if $\alpha,\beta\notin I$, $\alpha+\beta\in\Phi^+$, then $\alpha+\beta\notin I$.
\end{enumerate}
\end{proposition}

We can now introduce a restriction map, defined by Billey and Postnikov \cite{billey2005smoothness}. Let $E'\subset E$ be a subspace and $\Phi'=\Phi\cap E'$ is then a root system with an inherited set of positive roots $(\Phi')^+=\Phi^+\cap E'$. For any $w\in W(\Phi)$, its inversion set $I_{\Phi}(w)$ is biconvex and it is easy to see that the restriction $I_{\Phi}(w)\cap E'$ is also biconvex. By Proposition~\ref{prop:biconvex}, there is a unique element $w'\in W(\Phi')$ such that $I_{\Phi'}(w')=I_{\Phi}(w)\cap E'$. We call such $w'$ the \textit{restriction} of $w$ to $\Phi'$, denoted $w|_{\Phi'}$. 

\begin{definition}\label{def:BP}
We say that $w\in W(\Phi)$ contains the BP (Billey-Postnikov) pattern $\pi\in W(R)$, where choices of positive roots $\Phi^+\subset\Phi$ and $R^+\subset R$ have been fixed, if there exists a subspace $E'\subset E$ such that there is an isomorphism between root systems $\Phi':=\Phi\cap E$ and $R$ that preserves the chosen positive roots and maps $w|_{\Phi'}$ to $\pi$.
\end{definition}

We also introduce a new notion of \emph{linear patterns}, which enables an even nicer characterization of boolean elements.

\begin{definition}\label{def:linear}
We say that $w\in W(\Phi)$ contains the \textit{linear pattern} $\pi\in W(R)$, where choices of positive roots $\Phi^+\subset \Phi$ and $R^+\subset R$ have been fixed, if there exists a linear transformation $R\to \Phi$ that maps positive roots $R^+$ to positive roots $\Phi^+$, inversions $I_R(\pi)$ of $\pi$ to inversions $I_{\Phi}(w)$ of $w$, and non-inversions $R^+\setminus I_R(\pi)$ to non-inversions $\Phi^+\setminus I_{\Phi}(w)$. If the simple roots $\alpha_1, \ldots, \alpha_k$ of $R$ are mapped to $\beta_1, \ldots, \beta_k$, then we say that $w$ contains $\pi$ \emph{generated at} $\beta_1, \ldots, \beta_k$.
\end{definition}

We note that if $w$ contains the BP pattern $\pi$, then $w$ also contains the linear pattern $\pi$, but not necessarily the other way around. The difference between linear pattern containment and BP pattern containment is that in linear patterns, we do not require the map to be injective or angle-preserving, and we are also not required to map to all vectors in a subspace (we might map to only a strict subset of the vectors in a subspace). For example, there are linear patterns $\pi\in W(A_7)$ in $w\in W(E_7)$, and $\pi\in W(A_2)$ in $w\in W(B_2)$, but this is not the case for BP patterns. We proceed to give an example that demonstrates what linear patterns can look like.

\begin{example}
Let $\alpha_1$ be the long simple root of $B_2$ and $\alpha_2$ be the short simple root of $B_2$. Then $s_1 s_2 s_1\in W(B_2)$ contains the linear pattern $s_2 s_1 s_3 s_2\in W(A_3)$. Letting the simple roots of $A_3$ be $\beta_1, \beta_2, \beta_3$, this is demonstrated by sending $\beta_1\mapsto \alpha_2$, $\beta_2\mapsto \alpha_1$, $\beta_3\mapsto \alpha_2$. The rest of the map is then uniquely defined by linearity. As $I_{A_3}(s_2s_1s_3s_2)=\{\beta_2,\beta_1{+}\beta_2,\beta_2{+}\beta_3,\beta_1{+}\beta_2{+}\beta_3\}$, $(A_3)^+\setminus I_{A_3}(s_2s_1s_3s_2)=\{\beta_1,\beta_3\}$, $I_{B_2}(s_1s_2s_1)=\{\alpha_1,\alpha_1{+}\alpha_2,\alpha_1{+}2\alpha_2\}$, $(B_2)^+\setminus I_{B_2}(s_1s_2s_1)=\{\alpha_2\}$, we then see that inversions are sent to inversions, and non-inversions are sent to non-inversions. See Figure~\ref{fig:linear}. 
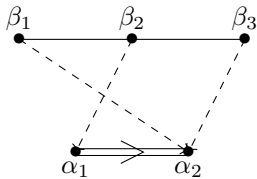
\begin{figure}[h!]
\centering
\begin{tikzpicture}[scale=1.5]
\node at (0,0) {$\bullet$};
\node[above] at (0,0) {$\beta_1$};
\node at (1,0) {$\bullet$};
\node[above] at (1,0) {$\beta_2$};
\node at (2,0) {$\bullet$};
\node[above] at (2,0) {$\beta_3$};
\draw(0,0)--(2,0);
\node at (0.5,-1) {$\bullet$};
\node[below] at (0.5,-1) {$\alpha_1$};
\node at (1.5,-1) {$\bullet$};
\node[below] at (1.5,-1) {$\alpha_2$};
\draw(0.5,-0.97)--(1.5,-0.97);
\draw(0.5,-1.03)--(1.5,-1.03);
\draw(0.9,-0.9)--(1.1,-1)--(0.9,-1.1);
\draw[dashed,->](0,0)--(1.5,-1);
\draw[dashed,->](1,0)--(0.5,-1);
\draw[dashed,->](2,0)--(1.5,-1);
\end{tikzpicture}
\caption{The linear  pattern $s_2 s_1 s_3 s_2\in W(A_3)$ in $s_1 s_2 s_1\in W(B_2)$}
\label{fig:linear}
\end{figure}
\end{example}

To further illustrate how BP patterns and linear patterns compare, we note that for $\pi$ in a type $A_k$ or $D_k$ Weyl group and $w$ in a type $A_n$ or $D_n$ Weyl group, $w$ contains the linear pattern $\pi$ iff $w$ contains the BP pattern $\pi$. We will not use this fact anywhere in the paper, and we only give a brief sketch of the proof. One can start by noting that in the case where both $R$ and $\Phi$ are irreducible and simply laced, the linear map in fact preserves angles, and further that for $A_n$ and $D_n$, it turns out that the linear pattern hits all the roots in the $\mathbb{R}$-span of the image, since there are no $A_n\subset D_n$ or $D_n\subset A_n$.

\section{Proof of the main theorem}\label{sec:mainpf}
We begin with the following simple proposition.
\begin{proposition}\label{prop:no-repeated-letters}
An element $w\in W$ is boolean if and only if any reduced expression (or equivalently, all reduced expressions) of $w$ does not contain repeated letters.
\end{proposition}
A special case of Proposition~\ref{prop:no-repeated-letters} appears as Proposition 7.3 in \cite{tenner2007pattern} in the case of finite classical types, with the proof omitted. 
\begin{proof}
If $w$ is boolean, then the interval $[\mathrm{id},w]$ has the same number of atoms as the height. The atoms of $[\mathrm{id},w]$ are the simple reflections used by any reduced expression of $w$ while the height is $\ell(w)$. This implies that any reduced expression cannot contain repeated letters. Conversely, if $w$ is a product of distinct simple reflections, $[\mathrm{id},w]$ being boolean follows directly from the subword property of the strong Bruhat order.
\end{proof}

We will first prove Theorem~\ref{thm:linear}. Then in Section~\ref{sub:lineartoBP}, we deduce the BP pattern version from Theorem~\ref{thm:linear}.

\subsection{Proof of Theorem~\ref{thm:linear}}
We start the proof with a very useful lemma.
\begin{lemma}\label{lem:si=support}
Let $w\in W(\Phi)$ and $\alpha\in\Delta$ be a simple root. Then any (or equivalently, all) reduced expression of $w$ contains $s_{\alpha}$ if and only if there exists $\beta\in I_{\Phi}(w)$ supported on $\alpha$.
\end{lemma}
\begin{proof}
Use induction on $\ell(w)$. The claim is clearly true when $\ell(w)=0$, where $w$ is the identity, with one reduced expression being the empty string and $I_{\Phi}(w)=\emptyset$. 

For the general case, assume first that there exists $\beta\in I_{\Phi}(w)$ supported on $\alpha$. We want to show that all reduced expressions of $w$ contain $s_{\alpha}$. A reduced expression of $w$ must end with $s_{\alpha'}$, where $\alpha'\in\Delta$ is a descent of $w$. If $\alpha'=\alpha$, we are done. If $\alpha'\neq\alpha$, by Lemma~\ref{lem:inv-after-si}, since $\beta\in I_{\Phi}(w)$, $s_{\alpha'}\beta\in I_{\Phi}(ws_{\alpha'})$. Since $\alpha\neq\alpha'$, $s_{\alpha'}\beta=\beta-\big(2\langle\alpha',\beta\rangle/\langle\alpha',\alpha'\rangle\big)\alpha'$ is also supported on $\alpha$. By induction hypothesis, all reduced expressions of $ws_{\alpha'}$ contain $s_{\alpha}$, so all reduced expressions of $w$ that end with $s_{\alpha'}$ contain $s_{\alpha}$. As we know this for every descent $\alpha'$, we know that all reduced expressions of $w$ contain $s_{\alpha}$. 

For the other direction, let $w=s_{i_1}\cdots s_{i_{\ell}}$ be a reduced expression and choose the largest $k$ such that $s_{i_k}=s_{\alpha}$. For $j=0,1,\ldots,\ell$, write $w^{(j)}=s_{i_1}\cdots s_{i_j}$ so that $w^{(0)}=\mathrm{id}$ and $w^{(\ell)}=w$. We use induction on $j$ from $k$ to $\ell$ to show that there exists $\beta_j\in I_{\Phi}(w^{(j)})$ such that $\beta_j$ is supported on $\alpha$. For $j=k$, take $\beta_k=\alpha$ since $\alpha$ is a descent for $w^{(k)}$. Now suppose we have $\beta_j$ constructed. By Lemma~\ref{lem:inv-after-si}, since $\beta_j\neq \alpha_{i_{j+1}}$, where $s_{i_{j+1}}$ denotes the reflection across the simple root $\alpha_{i_{j+1}}$, $s_{i_{j+1}}\beta_j\in I_{\Phi}(w^{(j+1)})$. We have $s_{i_{j+1}}\beta_j=\beta_j-\big(2\langle \alpha_{i_{j+1}},\beta_j\rangle/\langle \alpha_{i_{j+1}},\alpha_{i_{j+1}}\rangle\big)\alpha_{i_{j+1}}$ is supported on $\alpha$, since $\beta_j$ does and $\alpha_{i_{j+1}}\neq\alpha$ by maximality of $k$. Pick $\beta_{j+1}=s_{i_{j+1}}\beta_j$ and the induction step goes through. In the end, we conclude that there exists $\beta_{\ell}\in I_{\Phi}(w)$ that is supported on $\alpha$ as desired.
\end{proof}

\begin{remark}\label{rmk:si=reduced-typeA}
In the case of type $A_{n-1}$ where the Weyl group $W(A_{n-1})$ is isomorphic to the symmetric group $\S_n$, Lemma~\ref{lem:si=support} is saying that the simple transposition $s_k=(k\ k+1)$ appears in a reduced expression of $w$ if there exists $i,j$ such that $i\leq k<j$ and $w(i)>w(j)$. This fact can be easily observed.
\end{remark}

The following technical lemma, which is purely root-theoretic, is going to be important. It is also the only part of the proof that is not type-uniform. 
\begin{lemma}\label{lem:decompose}
Let $\alpha\in\Delta$ be a simple root and $\beta\neq\alpha\in\Phi^+$ be a positive root such that $s_{\alpha}\beta\in\Phi^+$ is supported on $\alpha$. Then (at least) one of the following is true:
\begin{enumerate}
\item $\beta+\alpha\in\Phi^+$;
\item $\beta=\alpha+\gamma_1+\gamma_2$ such that $\alpha+\gamma_1,\alpha+\gamma_2\in\Phi^+$ for some $\gamma_1,\gamma_2\in\Phi^+$;
\item $\beta=2\alpha+\gamma_1+\gamma_2+\gamma_3$ such that $\alpha+\gamma_i\in\Phi^+$ for $i\in\{1,2,3\}$, $\alpha+\gamma_i+\gamma_j\in\Phi^+$ for $i\neq j\in\{1,2,3\}$, and $\beta-\alpha\in\Phi^+$ for some $\gamma_1,\gamma_2,\gamma_3\in\Phi^+$.  
\end{enumerate}
\end{lemma}
\begin{proof}
Let us first reduce to the case where $\Phi$ is irreducible. We split $\beta=\beta_1+\beta_2$, where $\beta_1$ is the projection of $\beta$ to the span of the irreducible component containing the simple root $\alpha$, and $\beta_2$ is the projection of $\beta$ to the orthogonal complement. Note that the assumption of the lemma then holds for the pair $\alpha,\beta_1$. Then assuming the lemma in the irreducible case, we get that either (1) $\beta_1+\alpha\in \Phi^+$, in which case also $\beta+\alpha\in \Phi^+$; or (2) there is a decomposition $\beta_1=\alpha+\gamma_1+\gamma_2$, in which case we can also decompose $\beta=\alpha+\gamma_1+(\gamma_2+\beta_2)$; or (3) there is a decomposition $\beta_1=2\alpha+\gamma_1+\gamma_2+\gamma_3$, in which case we can also decompose $\beta=2\alpha+\gamma_1+\gamma_2+(\gamma_3+\beta_2)$. So it remains to prove the lemma for an irreducible root system.

For the classical types, we carry out a manual case check on the standard constructions. We will proceed type by type, starting from the simply laced types.
\begin{itemize}
    \item[Type $A_n$:] $\alpha$ is a simple root $e_i-e_{i+1}$, and $\beta$ is a positive root $e_j-e_k$ for some $j<k$. Keeping in mind that $s_\alpha(\beta)$ is supported on $\alpha$, there are a few options:
    \begin{enumerate}
        \item $j<i<i+1<k$. Then we decompose $\beta=(e_i-e_{i+1})+(e_j-e_i)+(e_{i+1}-e_k)$, as in (2).
        \item $j=i+1$. Then $\beta+\alpha\in \Phi^+$, as in (1).
        \item $k=i$. Then $\beta+\alpha\in \Phi^+$, as in (1).
    \end{enumerate}
    \item[Type $D_n$:] Due to the automorphism of the Dynkin diagram of $D_n$, we can assume that $\alpha=e_i-e_{i+1}$. If $\beta=e_j-e_k$, then we are in the type $A_{n-1}$ subsystem, so we are done by the type $A_n$ case (crucially, we use the fact that $\alpha$ is also a simple root of this $A_{n-1}$, and $\beta$ is supported on $\alpha$ when taken as a root of $A_{n-1}$). So this leaves us with the case $\beta=e_j+e_k$ (with $j<k$). We split into a few options for $\alpha$:
    \begin{itemize}
        \item $\alpha=e_{n-1}-e_n$. Then there are a few options for the indices, keeping in mind that $s_\alpha(\beta)$ is supported on $\alpha$.
        \begin{itemize}
            \item $k<n-1$. Then we decompose $\beta=(e_{n-1}-e_n)+(e_j+e_n)+(e_k-e_{n-1})$, as in (2).
            \item $j<n-1<k=n$. Then $\beta+\alpha\in \Phi^+$, as in (1).
        \end{itemize}
        \item $\alpha=e_i-e_{i+1}$ for $i<n-1$. We again split into cases for the indices.
        \begin{itemize}
            \item $j<i$. Split into cases again.
            \begin{itemize}
                \item $k\neq i, i+1$. Then we decompose $\beta=(e_i-e_{i+1})+(e_j-e_i)+(e_{i+1}+e_k)$, as in (2).
                \item $k=i+1$. Then $\beta+\alpha\in \Phi^+$, as in (1).
                \item $k=i$. Then we decompose $\beta=2(e_i-e_{i+1})+(e_j-e_i)+(e_{i+1}-e_n)+(e_{i+1}+e_n)$, as in (3).
            \end{itemize}
            \item $j=i, k=i+1$. Then we decompose $\beta=(e_i-e_{i+1})+(e_{i+1}-e_n)+(e_{i+1}+e_n)$, as in (2).
            \item $j=i+1$. Then $\beta+\alpha\in \Phi^+$, as in (1).
        \end{itemize}
    \end{itemize}
    
    \item[Type $B_n$:] If $\alpha,\beta$ are both in the type $A_{n-1}$ subsystem, then we are done by the type $A_n$ case, as before. It remains to consider the case where $\alpha=e_n$, $\beta=e_k$, or $\beta=e_j+e_k$ (with $j<k$). We proceed to check these cases.
    
    \begin{itemize}
        \item $\alpha=e_n$. There are a few options for $\beta$, again keeping in mind that $s_\alpha(\beta)$ is supported on $\alpha$.
        \begin{itemize}
            \item $\beta=e_k$. Then $\beta+\alpha\in \Phi^+$, as in (1).
            \item $\beta=e_j-e_n$. Then $\beta+\alpha\in \Phi^+$, as in (1).
            \item $\beta=e_j+e_k$ with $k<n$. Then we decompose $\beta=e_n+(e_k-e_n)+e_j$, as in (2).
        \end{itemize}
        \item $\beta=e_k$. Then the only case we have not considered yet is $\alpha=e_i-e_{i+1}$. Given that $s_\alpha(\beta)$ is supported on $\alpha$, there are a few options for the indices.
        \begin{itemize}
            \item $k<i$. Then we decompose $\beta=(e_i-e_{i+1})+(e_k-e_i)+e_{i+1}$, as in (2).
            \item $k=i+1$. Then we decompose $\beta+\alpha\in \Phi^+$, as in (1).
        \end{itemize}
        \item $\beta=e_j+e_k$. The remaining case is again $\alpha=e_i-e_{i+1}$. Keeping in mind that $s_\alpha(\beta)$ is supported on $\alpha$, there are a few options for the indices.
        \begin{itemize}
            \item $j<i$. We split into cases for $k$.
            \begin{itemize}
                \item $k\neq i, i+1$. Then we decompose $\beta=(e_i-e_{i+1})+(e_j-e_i)+(e_{i+1}+e_k)$, as in (2).
                \item $k=i+1$. Then $\beta+\alpha\in \Phi^+$, as in (1).
                \item $k=i$. Then we decompose $\beta=2(e_i-e_{i+1})+(e_j-e_i)+e_{i+1}+e_{i+1}$, as in (3).
            \end{itemize}
            \item $j=i, k=i+1$. Then we decompose $\beta=(e_i-e_{i+1})+e_{i+1}+e_{i+1}$, as in (2).
            \item $j=i+1$. Then $\beta+\alpha\in \Phi^+$, as in (1).
        \end{itemize}
    \end{itemize}
    \item[Type $C_n$:] Again, if $\alpha, \beta$ are both in the type $A_{n-1}$ subsystem, then we are done by the type $A_n$ case. It remains to consider the case where $\alpha= 2e_n,$ $\beta=2 e_k$, or $\beta=e_j+e_k$ (with $j<k$).
    
    \begin{itemize}
        \item $\alpha=2 e_n$. There are a few options for $\beta$, again keeping in mind that $s_\alpha(\beta)$ is supported on $\alpha$.
        \begin{itemize}
            \item $\beta=2e_k$. Then we decompose $\beta=2 e_n+(e_k-e_n)+(e_k-e_n)$, as in (2).
            \item $\beta=e_j-e_n$. Then $\beta+\alpha\in \Phi^+$, as in (1).
            \item $\beta=e_j+e_k$ with $k<n$. Then we decompose $\beta=2 e_n+(e_j-e_n)+(e_k-e_n)$, as in (2).
        \end{itemize}
        \item $\beta=2e_k$. The only case that is left is $\alpha=e_i-e_{i+1}$. Given that $s_\alpha(\beta)$ is supported on $\alpha$, there are a few options for the indices.
        \begin{itemize}
            \item $k<i$. Then we decompose $\beta=(e_i-e_{i+1})+(e_k-e_i)+(e_k+e_{i+1})$, as in (2).
            \item $k=i+1$. Then $\beta+\alpha\in \Phi^+$, as in (1).
        \end{itemize}
        \item $\beta=e_j+e_k$. The remaining case is again $\alpha=e_i-e_{i+1}$. Note that for $C_n$, $e_s-e_t$ is supported on $e_i-e_{i+1}$ iff $s\leq i$ and $i+1\leq t$, and that $e_s+e_t$ is supported on $e_i-e_{i+1}$ iff $s\leq i$. In fact, this condition is the same for $B_n$, so the cases for the indices $i,j,k$ that are possible are exactly the same as for the analogous case for $B_n$. We can even use all the same assignments to options (1), (2), and (3) as for $B_n$, except for cases where $e_s$ appears as $\alpha+\beta$ in option (1) or in a decomposition (in options (2) or (3)). One can go through the list of cases and confirm that this only happens twice. We now consider these two cases for $C_n$.
        \begin{itemize}
            \item $j<i,k=i$. Then we decompose $\beta=(e_i-e_{i+1})+(e_j-e_i)+(e_{i+1}+e_i)$, as in (2).
            \item $j=i,k=i+1$. Then $\beta+\alpha\in \Phi^+$, as in (1).
        \end{itemize}
    \end{itemize}
\end{itemize}

For the exceptional types $G_2$, $F_4$ and $E_8$, the lemma is checked on a computer. It is easy to check types $G_2$ and $F_4$ by hands but we won't do the tedious case analysis here. The cases of $E_6$ and $E_7$ follow from $E_8$, by identifying these as subsystems of $E_8$.
\end{proof}

We begin by proving that if $w$ contains one of our bad linear patterns, then it is not boolean. The fact that this proposition works so neatly is one of the main motivations for thinking about this in terms of linear patterns (instead of BP patterns).

\begin{proposition}\label{prop:onedir}
For irreducible $\Phi$, if $w\in W(\Phi)$ contains the linear pattern $s_1 s_2 s_1 \in W(A_2)$, $s_2 s_1 s_3 s_2 \in W(A_3)$, or $s_2 s_1 s_3 s_4 s_2 \in W(D_4)$, then $w$ is not boolean.
\end{proposition}

\begin{proof}
Say for a contradiction that $w$ contains $\pi$, one of these $3$ linear patterns, but is nevertheless boolean. We first claim that it suffices to show that for any simple root $\alpha$ which is an inversion of $w$, $w s_\alpha$ still contains the pattern $\pi$. To see this, note that $ws_\alpha$ is still boolean, since there is a reduced expression for $w$ ending in $s_\alpha$ (this is Corollary 1.4.6. in \cite{bjornerbrenti}), and then by induction, the identity Weyl group element contains the pattern $\pi$, which is a contradiction. Suppose $w$ contains the linear pattern $\pi$ generated at the positive roots $\beta_1, \ldots, \beta_k$, none of which is equal to $\alpha$. Then by Lemma~\ref{lem:inv-after-si}, $w s_\alpha$ contains $\pi$ generated at the positive roots $s_\alpha \beta_1, \ldots, s_\alpha \beta_k$. So the only way to get rid of $\pi$ is if $\beta_i=\alpha$ for at least one $i\in [k]$. It remains to do some casework.

\begin{itemize}
    \item If $\pi=s_1 s_2 s_1 \in W(A_2)$, let $\alpha'$ be the other root generating the linear pattern we are considering (in addition to $\alpha$). Then $\alpha+\alpha'\in \Phi^+$, so since $\Phi$ is irreducible, $\langle\alpha,\alpha'\rangle<0$, so $s_\alpha \alpha'=\alpha'-\left(2\langle \alpha, \alpha'\rangle/\langle\alpha, \alpha\rangle\right)\alpha$ is supported on $\alpha$. But $s_\alpha(\alpha')$ is an inversion of $w s_\alpha$, so by Lemma~\ref{lem:si=support}, any reduced expression of $w s_\alpha$ contains $s_\alpha$. However, this implies that there is a reduced expression for $w$ that contains two copies of $s_\alpha$, which is impossible.
    \item If $\pi= s_2 s_1 s_3 s_2\in W(A_3)$, then since $\alpha$ is an inversion, it must be the middle simple root $\beta_2$, as the middle simple root is the only simple root which is an inversion of $\pi$. In analogy to what we saw before, $\beta_1+\beta_2\in \Phi^+$ and $\beta_2+\beta_3\in \Phi^+$, so $\langle \beta_2, \beta_1\rangle$ and $\langle \beta_2,\beta_3\rangle<0$, from which we get that 
    \[s_\alpha(\beta_1+\beta_2+\beta_3)=\beta_1-\left(2\langle \alpha, \beta_1\rangle/\langle\alpha, \alpha\rangle\right)\alpha-\alpha+\beta_3-\left(2\langle \alpha, \beta_3\rangle/\langle\alpha, \alpha\rangle\right)\alpha\]
    
    is supported on $\alpha$. But it is also an inversion of $w s_\alpha$, so we get a contradiction as before.
    \item If $\pi=s_2 s_1 s_3 s_4 s_2\in W(D_4)$, then since $\alpha$ is an inversion, it must be $\beta_2$, since again this is the only simple root which is an inversion of $\pi$, and we again reach a contradiction because $s_\alpha(\beta_1+2\beta_2+\beta_3+\beta_4)$ is an inversion of $ws_{\alpha}$ which is supported on $\alpha$. 
\end{itemize}

\end{proof}

For the direction that $w$ avoids the $3$ bad patterns implies that $w$ is boolean, our main strategy is induction on the size of $\bigcup_{\beta\in I_\Phi(w)} \Supp(\beta)$ (the number of simple roots supporting some inversion of $w$) via the following technical lemma.
\begin{lemma}\label{lem:maintech}
If $w\in W(\Phi)$ avoids the $3$ bad patterns $s_1 s_2 s_1 \in W(A_2)$, $s_2 s_1 s_3 s_2 \in W(A_3)$, and $s_2 s_1 s_3 s_4 s_2 \in W(D_4)$, and $\alpha\in I_{\Phi}(w)$ is a simple root, then $I_{\Phi}(ws_{\alpha})$ contains no roots supported on $\alpha$ and moreover, $ws_{\alpha}$ does not contain any of these $3$ bad patterns.
\end{lemma}

\begin{proof}[Proof of Lemma~\ref{lem:maintech}]

Say for a contradiction that $w$ avoids the $3$ bad patterns and $\alpha\in I_{\Phi}(w)$ is a simple root, but there is a root $\gamma\in  I_{\Phi}(ws_{\alpha})$ supported on $\alpha$. We write $\beta=s_\alpha(\gamma)$, and note that $\gamma=s_\alpha(\beta)$, and that $s_\alpha(\beta)\in I_{\Phi}(ws_\alpha)\implies \beta\in I_\Phi(w)$  by Lemma~\ref{lem:inv-after-si}. By Lemma~\ref{lem:decompose} applied to these $\alpha,\beta$, we are now in one of the following three cases:

\begin{enumerate}
    \item $\alpha+\beta\in \Phi^+$. Note that $\alpha, \beta$ are inversions of $w$, and biconvexity implies that $\alpha+\beta$ is then also an inversion of $w$. So $w$ contains an $s_1 s_2 s_1\in W(A_2)$ generated at $\alpha, \beta$.
    \item $\beta=\gamma_1+\alpha+\gamma_2$ with $\gamma_1+\alpha, \alpha+\gamma_2\in \Phi^+$. Then if $\gamma_1$ or $\gamma_2$ is an inversion of $w$, $w$ contains an $s_1 s_2 s_1\in W(A_2)$ at respectively $\gamma_1, \alpha$ or $\alpha, \gamma_2$. If neither is an inversion, then we get that $\gamma_1+\alpha$ and $\alpha+\gamma_2$ are both inversions, since otherwise we would get a contradiction with biconvexity from $\gamma_1, \alpha+\gamma_2, \gamma_1+\alpha+\gamma_2$ or $\gamma_1+\alpha, \gamma_2, \gamma_1+\alpha+\gamma_2$. We have now determined whether all the relevant roots are inversions or non-inversions of $w$ to conclude that $w$ contains an $s_2 s_1 s_3 s_2\in W(A_3)$ generated at $\gamma_1, \alpha, \gamma_2$.
    \item $\beta=\gamma_1+2\alpha+\gamma_2+\gamma_3$. Then if $\gamma_1$, $\gamma_2$, or $\gamma_3$ is an inversion of $w$, $w$ contains a $s_1 s_2 s_1 \in W(A_2)$ at respectively $\gamma_1, \alpha$ or $\gamma_2, \alpha$ or $\gamma_3, \alpha$. We restrict to the remaining case that $\gamma_1, \gamma_2, \gamma_3$ are all non-inversions. If any of $\gamma_1+\alpha+\gamma_2, \gamma_1+\alpha+\gamma_3, \gamma_2+\alpha+\gamma_3$ is an inversion, then $w$ contains a bad pattern by case (2). 
    So we restrict to the case where these three roots are also non-inversions. Now if $\gamma_1+\alpha$ is a non-inversion, then we get a contradiction with biconvexity considering $\gamma_1+\alpha, \gamma_2+\alpha+\gamma_3$. So $\gamma_1+\alpha$ and analogously $\gamma_2+\alpha, \gamma_3+\alpha$ are inversions.
    Finally, biconvexity implies that $(\gamma_1+\alpha+\gamma_2)+\gamma_3$ is not an inversion. We have now determined whether all the relevant roots are inversions or non-inversions of $w$ to conclude that $w$ contains $s_2 s_1 s_3 s_4 s_2\in W(D_4)$.
\end{enumerate}

Given the first part of the lemma, we can now deduce the second part, i.e. that $w s_{\alpha}$ does not contain any of our $3$ bad patterns. Suppose it contains the bad pattern $\pi$ with simple roots mapping to $\beta_1, \ldots, \beta_k\in \Phi^+$ (where $k=2$, $k=3$, or $k=4$ depending on the pattern). If some $\beta_i=\alpha$, then we can note for each of our patterns that there is a root in $I_{\Phi}(ws_{\alpha})$ covering $\alpha$, which is impossible. So no $\beta_i$ is $\alpha$. But then it follows from Lemma~\ref{lem:inv-after-si} that $s_\alpha(\beta_1), \ldots, s_\alpha(\beta_k)$ generate a pattern $\pi$ in $w$, which is a contradiction. So $w s_{\alpha}$ also avoids the $3$ bad patterns.
\end{proof}

From here, the proof of the our linear pattern characterization of boolean elements (Theorem~\ref{thm:linear}) is short.

\begin{proof}[Proof of Theorem~\ref{thm:linear}]
Proposition~\ref{prop:onedir} gives one direction. As for the other direction, i.e. that $w$ which does not contain a bad pattern is boolean, we prove it by inducting on the size of $\bigcup_{\beta\in I_\Phi(w)} \Supp(\beta)$. The base case is trivial. As for the inductive step, we find a simple root $\alpha\in I_\Phi(w)$ (which exists e.g. by biconvexity), and consider $w s_\alpha$.  By Lemma~\ref{lem:maintech}, $\alpha\not\in \bigcup_{\beta\in I_\Phi(ws_\alpha)} \Supp(\beta)$. As multiplying a root by $s_\alpha$ only changes the coefficient of $\alpha$, all other simple roots in $\bigcup_{\beta\in I_\Phi(w)} \Supp(\beta)$ are also in $\bigcup_{\beta\in I_\Phi(ws_\alpha)} \Supp(\beta)$. Putting these observations together, we get that \[\left\lvert\bigcup_{\beta\in I_\Phi(ws_\alpha)} \Supp(\beta)\right\rvert=\left\lvert\bigcup_{\beta\in I_\Phi(w)} \Supp(\beta)\right\rvert-1.\]
By Lemma~\ref{lem:maintech}, $ws_\alpha$ also avoids the bad patterns. So by our inductive hypothesis, $ws_\alpha$ is boolean. We pick a reduced word for $ws_\alpha$. By Lemma~\ref{lem:si=support}, this reduced word does not contain $s_\alpha$. So if we add $s_\alpha$ to the end of this word, we get a word for $w$ in which each generator appear at most once. Since $s_\alpha$ is an inversion of $w$, this word is reduced. So $w$ is also boolean. This completes the induction.
\end{proof}

\subsection{From linear patterns to BP patterns}\label{sub:lineartoBP}
In this section, we deduce the characterization of boolean elements in terms of BP pattern avoidance (Theorem \ref{thm:main}) from the characterization in terms of linear pattern avoidance (Theorem \ref{thm:linear}). In the next lemma, we show that containing a linear pattern is equivalent to containing a corresponding set of BP patterns (each itself containing this linear pattern). After this lemma, most of this section is some casework on a finite number of Weyl groups to figure out the explicit sets of BP patterns that correspond to our linear patterns.

\begin{lemma}\label{lem:lintoBP}
Let $\Phi$ and $R$ be irreducible root systems, let $w\in W(\Phi)$ and $\pi \in W(R)$, and let $k$ be the rank of $R$. Then $w$ contains the linear pattern $\pi$ if and only if $w$ contains at least one BP pattern in the set

\[P_\pi:=\{\sigma \in \bigcup_{\substack{\Theta \text{ irreducible}\\ \text{rank}\left(\Theta\right)\leq k}}W(\Theta) \colon \sigma \text{ contains the linear pattern }\pi\}.\]

\end{lemma}

\begin{proof}
For the forward implication, restrict to the $\mathbb{R}$-span of the image of $R$ in $\Phi$. Denoting this subspace root system by $\Theta$, we define $\sigma=w|_{\Theta}$. Note that $\Theta$ has rank at most $k$, $\Theta$ is irreducible (this follows by considering the images of simple roots $\alpha_i$ of $R$ and deducing from $\alpha_i+\alpha_j\in \Phi^+\implies \langle \alpha_i,\alpha_j\rangle<0$ that they must all lie in the same irreducible component) and $\sigma$ contains the linear pattern $\pi$, so $\sigma\in P_{\pi}$. Hence, $w$ contains some BP pattern in $P_\pi$.

For the backward implication, suppose $w$ contains the BP pattern $\sigma\in P_\pi$, $\sigma\in W(\Theta)$. Then there is a linear map $R\to \Theta$ demonstrating that $\sigma$ contains the linear pattern $\pi$. Composing with the inclusion $\Theta\to \Phi$, we see that $\pi$ is also a linear pattern of $w$.
\end{proof}

Firstly, observe that $P_\pi$ is finite since there are only finitely many irreducible root systems of rank at most $k$, each having a finite Weyl group. Secondly, observe that it follows from the lemma that avoiding the linear patterns $\pi_1, \ldots, \pi_m$ is equivalent to avoiding all BP patterns in $P_{\pi_1}\cup \ldots\cup P_{\pi_m}$. Finally, observe that if there are $\sigma_1, \sigma_2\in P$ with $\sigma_1$ being a BP pattern in $\sigma_2$, then $w$ containing a BP pattern in $P$ is equivalent to $w$ containing a BP pattern in $P\setminus \{\sigma_2\}$. In other words, we can get rid of redundant elements. To make this precise, for any set $P$ or Weyl group elements, we define the \emph{reduction of $P$}, denoted $\text{red}(P)$, as:
\[\text{red}(P)=\{w\in P\colon w \text{ does not contain any BP pattern }\pi\neq w, \pi \in P\}.\]

With this notation, our observation is that avoiding all BP patterns in $P$ is equivalent to avoiding all BP patterns in $\text{red}(P)$.

For any two sets $P,S$ of Weyl group elements (which we think of as BP patterns), we also define the \emph{reduction of $P$ mod $S$}, denoted $P/S$, is the set of elements of $P$ which do not contain a BP pattern in $S$, i.e.
\[P/S:=\{w\in P\colon w \text{ does not contain any BP pattern } \pi \in S\}.\]

We now present $P_1:=P_{\pi_1}$ where $\pi_1=s_1 s_2 s_1\in W(A_2)$. We find this by just checking all elements of Weyl groups of irreducible root systems of rank at most $2$.

\begin{lemma}\label{lem:pi1}
The set of BP patterns corresponding to $\pi_1$, $P_1:=P_{\pi_1}$, consists of
\begin{itemize}
    \item $s_1 s_2 s_1\in W(A_2)$;
    \item $s_2 s_1 s_2, s_1 s_2 s_1 s_2 \in W(B_2)$;
    \item $s_2 s_1 s_2, s_1 s_2 s_1 s_2, s_2 s_1 s_2 s_1, s_1 s_2 s_1 s_2 s_1, s_2 s_1 s_2 s_1 s_2 , s_1 s_2 s_1 s_2 s_1 s_2 \in W(G_2)$.
\end{itemize}

\end{lemma}

\begin{proof}
Let us begin by observing that the linear map demonstrating containment of $\pi_1$ cannot send both simple roots to the same image, as $\beta\in \Phi^+\implies \beta+\beta=2\beta\not\in\Phi^+$.

We go through all the irreducible root systems of rank at most $2$. There are $4$ of these: $A_1,A_2,B_2, G_2$. By the observation above, there are no $w\in W(A_1)$ containing $\pi_1$.

For $A_2$, suppose the linear map demonstrating containment of $\pi_1$ takes the simple roots to $\beta_1,\beta_2$. Then by the observation above, the only option is that $\beta_1, \beta_2$ are the two simple roots of $A_2$. This determines the inversions as well (namely, $\beta_1,\beta_2,\beta_1+\beta_2$ are all inversions). The only $w\in W(A_2)$ with these inversions is $w=s_1 s_2 s_1$, which indeed contains $\pi_1$ (itself) as a pattern.

For $B_2$, let us call the simple roots $\alpha_1$ and $\alpha_2$. $W(B_2)$ has $8$ elements. $5$ of these have strictly fewer than $3$ inversions (these are $\mathrm{id}$ $s_1$, $s_2$, $s_1 s_2$, and $s_2 s_1$), and we can immediately conclude that these do not contain a linear $\pi_1$, since the observation above implies that $\beta_1,\beta_2,\beta_1+\beta_2$ are all distinct inversions. We check the $3$ remaining $w\in W(B_2)$:
\begin{itemize}
    \item $s_1 s_2 s_1\in W(B_2)$ has the three inversions $\alpha_1$, $\alpha_1+\alpha_2$, $\alpha_1+2\alpha_2$. No two of these add up to an inversion, so this does not contain a linear $\pi_1$.
    \item $s_2 s_1 s_2\in W(B_2)$ has the three inversions $\alpha_2, \alpha_1+\alpha_2, \alpha_1+2\alpha_2$. It contains a linear $\pi_1$ generated at $\alpha_2,\alpha_1+\alpha_2$.
    \item $s_1 s_2 s_1 s_2\in W(B_2)$ has all positive roots as inversions. It contains a linear $\pi_1$ generated at $\alpha_1,\alpha_2$.
\end{itemize}

This leaves us with $G_2$. Let us say that the simple roots of $G_2$ are $\alpha_1, \alpha_2$. $W(G_2)$ has $12$ elements, out of which $5$ we can rule out immediately on account of having strictly fewer than $3$ inversions -- these are $\mathrm{id}, s_1, s_2, s_1 s_2$, and $s_2 s_1$. Let us go through the rest:
\begin{itemize}
    \item $s_1 s_2 s_1$ has the $3$ inversions $2\alpha_1+3\alpha_2$, $\alpha_1+\alpha_2$, $\alpha_1$. No two of these add up to an inversion, so this does not contain a linear $\pi_1$.
    \item $s_2 s_1 s_2$ has the $3$ inversions $\alpha_1+2\alpha_2$, $\alpha_1+3\alpha_2$, $\alpha_2$. It contains a linear $\pi_1$ generated at $\alpha_1+2\alpha_2,\alpha_2$.
    \item $s_1 s_2 s_1 s_2$ has the $4$ inversions $2\alpha_1+3\alpha_2$, $\alpha_1+2\alpha_2$, $\alpha_1+3\alpha_2$, $\alpha_2$. It contains a linear $\pi_1$ generated at $\alpha_1+2\alpha_2,\alpha_2$.
    \item $s_2 s_1 s_2 s_1$ has the $4$ inversions $\alpha_1+2\alpha_2$, $2\alpha_1+3\alpha_2$, $\alpha_1+\alpha_2$, $\alpha_1$. It contains a linear $\pi_1$ generated at $\alpha_1+2\alpha_2,\alpha_1+\alpha_2$.
    \item $s_1 s_2 s_1 s_2 s_1$ has all positive roots other than $\alpha_2$ as inversions. It contains a linear $\pi_1$ generated at $\alpha_1+2\alpha_2, \alpha_1+\alpha_2$.
    \item $s_2 s_1 s_2 s_1 s_2$ has all positive roots other than $\alpha_1$ as inversions. It contains a linear $\pi_1$ generated at $\alpha_1+2\alpha_2,\alpha_1+\alpha_2$.
    \item $s_1 s_2 s_1 s_2 s_1 s_2$ has all positive roots as inversions. It contains a linear $\pi_1$ generated at $\alpha_1,\alpha_2$.
\end{itemize}

This completes the casework.

\end{proof}

We proceed to present $P_2:=\text{red}(P_{\pi_2})/P_{\pi_1}$ where $\pi_2=s_2 s_1 s_3 s_2\in W(A_3)$.

\begin{lemma}

The set of additional BP patterns corresponding to $\pi_2$, $P_2:=\text{red}(P_{\pi_2})/P_{\pi_1}$, consists of
\begin{itemize}
    \item $s_1 s_2 s_1 \in W(B_2)$;
    \item $s_2 s_1 s_3 s_2\in W(A_3)$;
    \item $s_2 s_1 s_3 s_2\in W(C_3)$.
\end{itemize}

\end{lemma}

\begin{proof}
$\text{red}(P_{\pi_2})/P_{\pi_1}$ consists of all elements of Weyl groups of root systems of rank at most $3$ that contain $\pi_2$ and do not contain the linear pattern $\pi_1$ nor any smaller BP pattern that contains the linear pattern $\pi_2$. The root systems of rank at most $3$ are $A_1, A_2, B_2, G_2, A_3, B_3, C_3$. Suppose that a linear map demonstrating containment of $\pi_2$ sends the simple roots to $\beta_1,\beta_2,\beta_3$. We note that then $\beta_1+\beta_2+\beta_3$ has to be a root. This already implies that a linear $\pi_2$ is not contained in any element of $W(A_1)$ or $W(A_2)$. Also note that $\beta_2, \beta_2+\beta_3$, and $\beta_1+\beta_2+\beta_3$ are all distinct inversions, so if $w$ contains $\pi_2$, then $w$ must have at least $3$ inversions, and that these inversions are all $\geq\beta_2$ in the root poset (this last fact will be useful later). We can use this inversion count to check the case of $B_2$. Consider the argument for $B_2$ in the proof of Lemma~\ref{lem:pi1}. We note that the same elements are ruled out on account of not having enough inversions. The only element which has not been ruled out and also does not contain $\pi_1$ is then $s_1 s_2 s_1\in W(B_2)$. It has the inversions $\alpha_1,\alpha_1+\alpha_2,\alpha_1+2\alpha_2$ (here and later, we are letting $\alpha_i$ be the simple roots of the root system into which we are considering a linear map, ordered according to our conventions), so it contains $\pi_2$ generated at $\alpha_2, \alpha_1, \alpha_2$. There is no strict subspace which contains a linear $\pi_2$, so we conclude that $s_1 s_2 s_1\in \text{red}(P_{\pi_2})/P_{\pi_1}$

For $G_2$, the only element left to consider after the proof of Lemma~\ref{lem:pi1} and counting inversions is $s_1 s_2 s_1\in W(G_2)$, which has the inversions $2\alpha_1+3\alpha_2, \alpha_1+\alpha_2,\alpha_1$. If $\beta_2\neq \alpha_1$ in this case, then it is not possible for there to be three distinct inversions containing $\beta_2$. Hence, we just need to consider the case that $\beta_2=\alpha_1$. Arguing similarly, we get that $\beta_2+\beta_3=\alpha_1+\alpha_2\implies \beta_3=\alpha_2$. Continuing, $\beta_1+\beta_2+\beta_3=2\alpha_1+3\alpha_2\implies \beta_1=\alpha_1+2\alpha_2$. However, $\beta_1+\beta_2=\alpha_1+3\alpha_2$ is not an inversion, so this does not give a linear $\pi_2$. Hence, we get no new elements from $G_2$.

For $A_3$, coefficient counting gives that $\beta_1, \beta_2,\beta_3$ must all be simple roots, from which $\beta_1+\beta_2,\beta_2+\beta_3\in \Phi^+\implies \beta_2=\alpha_2$, and WLOG $\beta_1=\alpha_1, \beta_3=\alpha_3$. The inversions of $w\in W(A_3)$ containing a linear $\pi_2$ pattern are fully determined by this, and this determines $w$ to be $s_2 s_1 s_3 s_2 \in W(A_3)$. One can check that this does not contain a linear $\pi_1$ and also that there is no strict subspace which contains $\pi_2$, so $s_2 s_1 s_3 s_2\in \text{red}(P_{\pi_2})/P_{\pi_1}$

Let us now show that $\text{red}(P_{\pi_2})/P_{\pi_1}$ contains no elements of $W(B_3)$. Suppose we have $w\in W(B_3)$ containing a linear $\pi_2$. Let the corresponding linear map send the simple roots of $A_3$ to $\beta_1, \beta_2, \beta_3\in B_3$. Let us consider the options for $\beta_1, \beta_2, \beta_3$. If the $\mathbb{R}$-span of these three roots is a proper subspace of the ambient $B_3$, then the restriction of $w$ to the root system in this subspace contains $\pi_2$, so $w\not\in \text{red}(P_{\pi_2})$. This leaves us with the case where the $\mathbb{R}$ span of $\beta_1, \beta_2, \beta_3$ is full-dimensional. Note that $\beta_1+\beta_2+\beta_3\in B_3$, so its height (the sum of its coefficients in the basis of simple roots $\alpha_1, \alpha_2, \alpha_3$ of $B_3$) is at most $5$. It is also at least $3$, and we will consider each option:
\begin{itemize}
    \item If the height of $\beta_1+\beta_2+\beta_3$ is $3$, the only full-dimensional option that also satisfies the condition that $\beta_1+\beta_2, \beta_2+\beta_3\in \Phi^+$ is $\beta_1=\alpha_1, \beta_2=\alpha_2, \beta_3=\alpha_3$ (the map in reverse order is also possible, but this gives the same inversions since $\pi_2$ is preserved under the isomorphism of $A_3$). Then $\beta_1+\beta_2+\beta_3$, $\beta_2+\beta_3$, and (by biconvexity) $\beta_1+2\beta_2+2\beta_3$ are all inversions, so $w$ contains a linear $\pi_1$ generated at $\beta_1+\beta_2+\beta_3,\beta_2+\beta_3$, so $w$ is not in $\text{red}(P_{\pi_2})/P_{\pi_1}$. 
    
    \item If the height of $\beta_1+\beta_2+\beta_3$ is $4$, then one of $\beta_1,\beta_2,\beta_3$ has height $2$, and the other two have height $1$ each. $B_3$ only has two roots of height $2$, namely $\alpha_1+\alpha_2$ and $\alpha_2+\alpha_3$. If the height two root is $\alpha_1+\alpha_2$, then any root adjacent to its preimage in the $A_3$ Dynkin diagram has to be sent to $\alpha_3$ (since neither $\alpha_1+(\alpha_1+\alpha_2)$ nor $(\alpha_1+\alpha_2)+\alpha_2$ is a root). Since we are in the full-dimensional case, this implies that $\beta_2\neq \alpha_1+\alpha_2$, so WLOG (as before) $\beta_1=\alpha_1+\alpha_2$. From what we already argued, it follows that $\beta_2=\alpha_3$. The only option for $\beta_3$ is $\beta_3=\alpha_2$. But then $w$ contains a linear $\pi_1$ generated at $\alpha_2+\alpha_3, \alpha_3$. So $w$ is not in $\text{red}(P_{\pi_2})/P_{\pi_1}$. If instead the height $2$ root is $\alpha_2+\alpha_3$, then an adjacent root can be $\alpha_1$ or $\alpha_3$. If $\alpha_2+\alpha_3=\beta_2$, then WLOG $\beta_1=\alpha_1$ and $\beta_3=\alpha_3$. Then $w$ contains a linear $\pi_1$ generated at $\alpha_1+\alpha_2+\alpha_3,\alpha_2+\alpha_3$. So $w$ is not in $\text{red}(P_{\pi_2})/P_{\pi_1}$. The remaining case is that WLOG $\alpha_2+\alpha_3=\beta_1$, in which case either $\beta_2=\alpha_1$, and hence $\beta_3=\alpha_2$; or $\beta_2=\alpha_3$, and hence $\beta_3=\alpha_2$. The former case is impossible since then $\beta_1+\beta_2+\beta_3\not\in \Phi^+$. In the latter case, we find a linear $\pi_1$ generated at $\alpha_2+\alpha_3, \alpha_3$. So $w$ is not in $\text{red}(P_{\pi_2})/P_{\pi_1}$.
    
    \item If the height of $\beta_1+\beta_2+\beta_3$ is $5$, let us split into cases according to whether one of $\beta_1,\beta_2,\beta_3$ has height $3$.
    
    If one of $\beta_1,\beta_2,\beta_3$ has height $3$, then this can be either $\alpha_1+\alpha_2+\alpha_3$ or $\alpha_2+2\alpha_3$ (since $B_3$ has no other roots of height $3$). In either case, $\beta_1+\beta_2+\beta_3=\alpha_1+2\alpha_2+2\alpha_3$ (since this is the unique root in $B_3$ of height $5$), and this lets us determine the other two $\beta_i$. In the former case, the other two are $\alpha_2,\alpha_3$, in which case we note by considering pairwise sums that the only option is $\beta_1=\alpha_1+\alpha_2+\alpha_3$, $\beta_2=\alpha_3, \beta_3=\alpha_2$. But then $w$ contains a linear $\pi_1$ generated at $\alpha_2+\alpha_3, \alpha_3$. So $w$ is not in $\text{red}(P_{\pi_2})/P_{\pi_1}$. In the latter case, the other two are $\alpha_1,\alpha_2$, in which case (arguing as before) we get $\beta_1=\alpha_2+2\alpha_3, \beta_2=\alpha_1, \beta_3=\alpha_2$. But then there is a linear $\pi_2$ generated at $\alpha_3, \alpha_1+\alpha_2, \alpha_3$ which is contained in a $2$-dimensional subspace, so $w\not\in \text{red}(P_{\pi_2})$.
    
    If there is no root among $\beta_1,\beta_2,\beta_3$ of height $3$, then two have height $2$ and one has height $1$. The only two roots of $B_3$ of height $2$ are $\alpha_1+\alpha_2$ and $\alpha_2+\alpha_3$, and since it is not possible that one of these appears twice (that would contradict with full-dimensionality), both must appear once among $\beta_1,\beta_2,\beta_3$. Since $(\alpha_1+\alpha_2)+(\alpha_2+\alpha_3)\not\in \Phi^+$, these must be $\beta_1$ and $\beta_3$, so WLOG $\beta_1=\alpha_1+\alpha_2, \beta_3=\alpha_2+\alpha_3$. Using $\beta_1+\beta_2+\beta_3=\alpha_1+2\alpha_2+2\alpha_3$, we get that $\beta_2=\alpha_3$. But then $w$ contains a linear $\pi_1$ generated at $\alpha_1+\alpha_2, \alpha_2+2\alpha_3$. So $w$ is not in $\text{red}(P_{\pi_2})/P_{\pi_1}$.
    
\end{itemize}

This completes the case check for $B_3$. We do the same for $C_3$. The case check can be set up completely analogously, and there will again be $3$ options for the height of $\beta_1+\beta_2+\beta_3$. We will now find that there is exactly one $w\in W(C_3)$ that contains a linear $\pi_2$.
\begin{itemize}
    \item If the height of $\beta_1+\beta_2+\beta_3$ is $3$, then again the only option we have to consider is $\beta_1=\alpha_1, \beta_2=\alpha_2, \beta_3=\alpha_3$. Then $\alpha_2$, $\alpha_2+\alpha_3$, and (by biconvexity) $2\alpha_2+\alpha_3$ are all inversions, so we find a linear $\pi_1$.
    \item If the height of $\beta_1+\beta_2+\beta_3$ is $4$, then the height $2$ $\beta_i$ is $\alpha_1+\alpha_2$ or $\alpha_2+\alpha_3$. If the height $2$ $\beta_i$ is $\alpha_1+\alpha_2$, then arguing like for $B_3$, we get WLOG $\beta_1=\alpha_1+\alpha_2$, $\beta_2=\alpha_3$, $\beta_3=\alpha_2$. Then note that $\beta_1+\beta_2+\beta_3=\alpha_1+2\alpha_2+\alpha_3$ is an inversion, so by biconvexity, at least one of $\alpha_1$ and $2\alpha_2+\alpha_3$ is an inversion. If the former is an inversion, there is a linear $\pi_1$ generated at $\alpha_1, \alpha_2+\alpha_3$. If the latter is an inversion, there is a linear $\pi_2$ generated at $\alpha_2, \alpha_3, \alpha_2$, the span of which is a proper subspace. This leaves us with the case that the height $2$ $\beta_i$ is $\alpha_2+\alpha_3$. If $\alpha_2+\alpha_3=\beta_2$, then WLOG $\beta_1=\alpha_1$, $\beta_3=\alpha_2$. Since $\alpha_2+\alpha_3$ is an inversion, biconvexity gives that at least one of $\alpha_2,\alpha_3$ is an inversion. If $\alpha_2$ is an inversion, then we have a $\pi_1$ generated at $\alpha_2, \alpha_2+\alpha_3$. If $\alpha_3$ is an inversion, then we have a $\pi_2$ generated at $\alpha_2,\alpha_3,\alpha_2$, the span of which is a proper subspace. The remaining case is that WLOG $\alpha_2+\alpha_3=\beta_1$; then $\beta_2=\alpha_1$ or $\beta_2=\alpha_2$. In the case that $\beta_2=\alpha_1$, we get that $\beta_3=\alpha_2$, from which $\beta_1+\beta_2=\alpha_1+\alpha_2+\alpha_3$ and $\beta_2+\beta_3=\alpha_1+\alpha_2$ are inversions, so there is a $\pi_1$ generated at $\alpha_1+\alpha_2, \alpha_1+\alpha_2+\alpha_3$. In the case that $\beta_2=\alpha_2$, we get that $\beta_3=\alpha_1$ (the case $\beta_3=\alpha_3$ is ruled out since $\beta_1+\beta_2+\beta_3\in \Phi^+$). Then note that $\alpha_2, 2\alpha_2+\alpha_3, \alpha_1+\alpha_2,\alpha_1+2\alpha_2+\alpha_3$ are inversions and $\alpha_1,\alpha_2+\alpha_3$ are non-inversions (by definition of linear pattern containment). The remaining roots are $\alpha_3,\alpha_1+\alpha_2+\alpha_3,2\alpha_1+2\alpha_2+\alpha_3$. If $\alpha_3$ is an inversion, then there is a linear $\pi_1$ generated at $\alpha_2,\alpha_3$. If $\alpha_1+\alpha_2+\alpha_3$ is an inversion, then there is a linear $\pi_1$ generated at $\alpha_1+\alpha_2+\alpha_3,\alpha_2$. If $2\alpha_1+2\alpha_2+\alpha_3$ is an inversion, then there is a linear $\pi_2$ generated at $\alpha_1,2\alpha_2+\alpha_3,\alpha_1$. So the only option is that $\alpha_3,\alpha_1+\alpha_2+\alpha_3,2\alpha_1+2\alpha_2+\alpha_3$ are all non-inversions, in which case the inversion set is exactly known and determines $w$ to be $s_2 s_1 s_3 s_2\in W(C_3)$. One can check explicitly that it does not contain any $\pi_1$, nor does it contain any $\pi_2$ in a strict subspace, so $s_2 s_1 s_3 s_2\in \text{red}(P_{\pi_2})/P_{\pi_1}$. We have now checked all options for the height $4$ case.
    \item If the height of $\beta_1+\beta_2+\beta_3$ is $5$, we again split into cases according to whether some $\beta_i$ has height $3$. 
    
    If one of $\beta_1,\beta_2,\beta_3$ has height $3$, then this can be either $\alpha_1+\alpha_2+\alpha_3$ or $2\alpha_2+\alpha_3$. Either way, $\beta_1+\beta_2+\beta_3=2\alpha_1+2\alpha_2+\alpha_3$, from which we can deduce the other two $\beta_i$. In the case that the height $3$ root is $\alpha_1+\alpha_2+\alpha_3$, we get the other two to be $\alpha_1$ and $\alpha_2$, from which WLOG $\beta_1=\alpha_1$, $\beta_2=\alpha_2$, $\beta_3=\alpha_1+\alpha_2+\alpha_3$. Note that $\beta_1+\beta_2+\beta_3=2\alpha_1+2\alpha_2+\alpha_3$ is an inversion but $\beta_3=\alpha_1+\alpha_2+\alpha_3$ is not an inversion, from which biconvexity gives that $\alpha_2+\alpha_3$ is an inversion. But then there is a $\pi_1$ generated at $\alpha_2,\alpha_2+\alpha_3$. In the case that the height $3$ root is $2\alpha_2+\alpha_3$, we get that the other two are $\alpha_1$ and $\alpha_1$. However, these span a two-dimensional subspace. 
    
    The remaining option is that two of $\beta_1,\beta_2,\beta_3$ have height $2$. The only two height $2$ roots are $\alpha_1+\alpha_2$ and $\alpha_2+\alpha_3$. By considering the dimension of the span, we see that it is not possible for one of these to appear twice, so both must be some $\beta_i$, which implies that the third root is $(2\alpha_1+2\alpha_2+\alpha_3)-(\alpha_1+\alpha_2)-(\alpha_2+\alpha_3)=\alpha_1$. By considering pairwise sums, we see that WLOG $\beta_1=\alpha_1, \beta_2=\alpha_2+\alpha_3, \beta_3=\alpha_1+\alpha_2$. Since $\alpha_2+\alpha_3$ is an inversion, biconvexity gives that at least one of $\alpha_2,\alpha_3$ is an inversion. If $\alpha_2$ is an inversion, then we find a linear $\pi_1$ generated at $\alpha_2, \alpha_1+\alpha_2+\alpha_3$. If $\alpha_2$ is not an inversion, then $\alpha_3$ is an inversion. Also note that $\alpha_1+2\alpha_2+\alpha_3$ is an inversion but $\alpha_1$ is not an inversion, so biconvexity implies that $2\alpha_2+\alpha_3$ is an inversion. So we find a linear $\pi_2$ generated at $\alpha_2,\alpha_3,\alpha_2$.

\end{itemize}

\end{proof}

We proceed to present $P_3:=\text{red}(P_{\pi_3})/(P_{\pi_1}\cup P_{\pi_2})$ where $\pi_3=s_2 s_1 s_3 s_4 s_2\in W(D_4)$.

\begin{lemma}

The set of additional BP patterns corresponding to $\pi_3$, $P_3=\text{red}(P_{\pi_3})/(P_{\pi_1}\cup P_{\pi_2})$, consists of
\begin{itemize}
    \item $s_1 s_2 s_1\in W(G_2)$;
    \item $s_2 s_1 s_3 s_2 \in W(B_3)$;
    \item $s_2 s_1 s_3 s_4 s_2\in W(D_4)$.
\end{itemize}
\end{lemma}

One can prove this analogously to what we did for $\pi_1$ and $\pi_2$, i.e., by checking all elements of Weyl groups of irreducible root systems of rank at most $4$ (which we did using computer assistance), but we skip this.

We now give the proof of Theorem \ref{thm:main} from \ref{thm:linear}.

\begin{proof}[Proof of Theorem \ref{thm:main}]
We continue with the notation $\pi_1=s_1 s_2 s_1\in W(A_2)$, $\pi_2=s_2 s_1 s_3 s_2\in W(A_3)$, and $\pi_3= s_2 s_1 s_3 s_4 s_2\in W(D_4)$. By Theorem \ref{thm:linear}, for an irreducible root system $\Phi$, $w\in W(\Phi)$ is boolean iff it avoids the linear patterns $\pi_1, \pi_2, \pi_3$. By lemma \ref{lem:lintoBP} and the observations after it, $w$ avoids linear $\pi_1, \pi_2,\pi_3$ iff $w$ avoids all BP patterns in $P_1\cup P_2\cup P_3=:P$, which is exactly the set of BP patterns given in Theorem \ref{thm:main}. Hence, for $\Phi$ irreducible, $w\in W(\Phi)$ is boolean iff it avoids all BP patterns in $P$. Note that for any (not necessarily irreducible) root system $\Phi$, $w\in W(\Phi)$ is boolean iff the restrictions of $w$ to all irreducible components of $\Phi$ are boolean. Note that if $w$ avoids all the BP patterns in $P$, then the restriction of $w$ to any irreducible component also avoids all BP patterns in $P$, and hence is boolean. Conversely, if $w$ contains a BP pattern in $P$, then note that since all these patterns live in irreducible root systems, the subspace in which such a pattern is found is a subspace of an irreducible component, and hence the restriction of $w$ to this irreducible component is not boolean. These two directions together show that $w\in W(\Phi)$ is boolean iff it avoids all BP patterns in $P$. This is the statement of \ref{thm:main}, which is what we wanted to prove.
\end{proof}

\section{$k$-Boolean permutations}\label{sec:k-boolean}
Inspired by Proposition~\ref{prop:no-repeated-letters}, we define $k$-boolean permutations. 
\begin{definition}
A permutation $w\in\S_n$ is $k$-\textit{boolean} if for any reduced word of $w$, there is no simple transposition $s_i$ that appears strictly more than $k$ times.
\end{definition}

We see that $w$ is 0-boolean if and only if $w$ is the identity. Also by definition, being 1-boolean is the same as being boolean.

\begin{theorem}\label{thm:2boolean}
A permutation $w\in\S_n$ is 2-boolean if and only if $w$ avoids 3421, 4312, 4321 and 456123.
\end{theorem}

\begin{remark}
It is clear that being 0-boolean is equivalent to avoiding the pattern 21, and we know that being 1-boolean and being 2-boolean are characterized by pattern avoidance as well. However, it is not true that being $k$-boolean for $k\geq3$ is characterized by pattern avoidance. We have that $436512=s_3s_2s_3s_4s_5s_1s_2s_3s_4s_3$ is not 3-boolean since this reduced expression contains 4 copies of $s_3$. However, 4357612, which contains 436512 as a pattern, is 3-boolean by a computer check.
\end{remark}
We prove Theorem~\ref{thm:2boolean} in Section~\ref{sub:proof2boolean} and we then enumerate them in Section~\ref{sub:enumeration2boolean}.

\subsection{Proof of Theorem~\ref{thm:2boolean}}\label{sub:proof2boolean}
This section is devoted to proving Theorem~\ref{thm:2boolean}, which boils down to tedious case checking. We split the proof into two halves, one for each direction.
\begin{lemma}\label{lem:bad-not2boolean}
If a permutation $w\in\S_n$ contains 3421, 4312, 4321 or 456123, then there exists a reduced expression $w=s_{i_1}\cdots s_{i_{\ell}}$ where some simple transposition $s_k=(k\ k+1)$ appears at least 3 times.
\end{lemma}
\begin{proof}
Multiplying $w$ by $s_k$ on the right can be thought of as swapping the values at index $k$ and $k+1$, in one-line notation of the permutation. We are going to construct a reduced expression of $w$ by using the simple transposition to gradually reduce the length of $w$ until we obtain the identity permutation. \\

\noindent\textbf{Case 1:} $w$ contains 3421. Suppose that $w$ contains 3421 at indices $r_1<r_2<r_3<r_4$ with $w(r_1)=c$, $w(r_2)=d$, $w(r_3)=b$, $w(r_4)=a$ with $a<b<c<d$. We pick $(r_1,r_2,r_3,r_4)$ such that $r_3-r_2$ is as small as possible. In this way, for every $j$ in the range of $r_2<j<r_3$, if $w(j)>c$, then we can replace $r_2$ by $j$ to decrease $r_3-r_2$, contradicting its minimality and if $a<w(j)<c$, we can replace $r_3$ by $j$ to decrease $r_3-r_2$, contradicting its minimality as well. As a result, for $r_2<j<r_3$, we must have $w(j)<a$. Let $k=r_3-1$ and we will show that we can use $s_k$ at least 3 times to decrease $w$ down to the identity. 

First, let $w^{(1)}=ws_{r_2}s_{r_2+1}\cdots s_{r_3-2}$ where the length of $w$ is decreasing by 1 at each step. We then have $w^{(1)}(k)=d$ and $w^{(1)}(k+1)=b$, which form a descent. 

Let $w^{(2)}=w^{(1)}s_k$. 

Now we multiply $w^{(2)}$ by some products of $s_i$'s to obtain $w^{(3)}$, where $k+1\leq i\leq r_4-1$ to sort the indices $\{k+1,k+2,\ldots,r_4\}$, i.e., $w^{(3)}(k+1)<\cdots<w^{(3)}(r_4)$ and $\{w^{(3)}(k+1),\ldots,w^{(3)}(r_4)\}=\{w^{(2)}(k+1),\ldots,w^{(2)}(r_4)\}$, while decreasing the length of $w$ by 1 in each step. 

We observe that $w^{(3)}(k)=b$ and $w^{(3)}(k+1)=a'\leq w^{(2)}(r_4)=a$ so let $w^{(4)}=w^{(3)}s_k$. Finally, notice that $w^{(4)}(r_1)=c>w^{(4)}(k+1)=b$, which means $w^{(4)}$ has an inversion supported on $s_k$. By Lemma~\ref{lem:si=support} (Remark~\ref{rmk:si=reduced-typeA}), any reduced expression of $w^{(4)}$ contains $s_k$. We have thus obtained three copies of $s_k$.

A diagram of the above steps is shown in Figure~\ref{fig:3421}.
\begin{figure}[h!]
\centering
\begin{tikzpicture}[scale=0.5]
\node at (0,0) {$
\begin{aligned}
w=&\cdots c\cdots d\cdots b\cdots\cdots a\cdots\\
w^{(1)}=&\cdots c\cdots\cdots db\cdots\cdots a\cdots\\
w^{(2)}=&\cdots c\cdots\cdots bd\cdots\cdots a\cdots\\
w^{(3)}=&\cdots c\cdots\cdots ba'\cdots d\cdots\cdots\\
w^{(4)}=&\cdots c\cdots\cdots a'b\cdots d\cdots\cdots
\end{aligned}
$};
\draw[->](6,1.2)--(6,0);
\node[right] at (6,0.6) {$s_k$};
\draw[->](6,-1.1)--(6,-2.3);
\node[right] at (6,-1.7) {$s_k$};
\draw(-1.4,-2.6)--(-1.4,-2.9)--(1.3,-2.9)--(1.3,-2.6);
\node at (0,-3.5) {\small supported on $s_k$};
\end{tikzpicture}
\caption{3421 implies some $s_k$ appearing at least 3 times}
\label{fig:3421}
\end{figure}

\noindent\textbf{Case 2:} $w$ contains 4312. Since 4312 is the inverse of 3421, this case follows from Case 1 by taking inverse.\\

\noindent\textbf{Case 3:} $w$ contains 4321. This is a simpler version of Case 1. As we have already done Case 1, we may as well assume that $w$ avoids 3421. Suppose that $w$ contains 4321 at indices $r_1<r_2<r_3<r_4$ with $w(r_1)=d$, $w(r_2)=c$, $w(r_3)=b$, $w(r_4)=a$ with $a<b<c<d$. For $j$ in the range of $r_2<j<r_3$, we must have $w(j)<w(r_2)=c$ since otherwise, $w$ contains 3421 at indices $r_2<j<r_3<r_4$. We can now run exactly the same argument as in Case 1 by switching all $c$'s with $d$'s. One could also refer to Figure~\ref{fig:3421} by considering $c$ and $d$ swapped.\\

\noindent\textbf{Case 4:} $w$ contains 456123. The argument is also largely similar. Suppose that $w$ contains 456123 at indices $r_1<\cdots<r_6$ with $w(r_1)=d$, $w(r_2)=e$, $w(r_3)=f$, $w(r_4)=a$, $w(r_5)=b$ and $w(r_6)=c$. Let $r_3\leq k<r_4$ be any index in between $r_3$ and $r_4$. Consider $w^{(1)}$, which is obtained from $w$ by sorting indices $r_3,r_3+1,\ldots,r_4$ in order, i.e. $w^{(1)}(r_3)<\cdots<w^{(1)}(r_4)$. Equivalently, we can obtain $w^{(1)}$ from $w$ by multiplying $s_j$ on the right, for some $r_3\leq j<r_4$, so that the length decreases after the multiplication, until such operation cannot be performed anymore. By Lemma~\ref{lem:si=support} (Remark~\ref{rmk:si=reduced-typeA}), as $w(r_3)>w(r_4)$, $s_k$ must be used. Next, let $w^{(2)}$ be the permutation obtained from $w^{(1)}$ by sorting indices $r_2,\ldots,r_5$. Similarly, as $w^{(1)}(r_2)>w^{(1)}(r_5)$, $s_k$ is used in the process. Finally, let $w^{(3)}$ be the permutation obtained from $w^{(2)}$ by sorting indices $r_1,\ldots,r_6$ and as $w^{(2)}(r_1)>w^{(2)}(r_6)$, $s_k$ is used a third time.
\end{proof}

We now proceed to the other direction of Theorem~\ref{thm:2boolean}.
\begin{lemma}\label{lem:inclength-stillbad}
Let $w$ be a permutation that contains one of 3421, 4312, 4321, 456123. If $u=ws_k$, or $u=s_kw$, such that $\ell(u)=\ell(w)+1$, then $u$ also contains one of these patterns.
\end{lemma}
\begin{proof}
Let's note that the set of patterns of interest is closed under taking inverses, so it suffices to consider only the case $u=ws_k$. Assume that $w$ contains $\pi$, one of the pattern of interst, at indices $r_1<\cdots<r_m$, where $m\in\{4,6\}$. If $\{k,k+1\}\cap\{r_1,\ldots,r_m\}\leq1$, then $u=ws_k$ contains the same pattern $\pi$. If $\{k,k+1\}\subset\{r_1,\ldots,r_m\}$, then $u$ contains $\pi s_j$, for some $s_j$ such that $\ell(\pi s_j)=\ell(\pi)+1$. If $\pi=3421$, then we must have $j=1$ and $\pi s_j=4321$; if $\pi=4312$, then $j=3$ and $\pi s_j=4321$; if $\pi=4321$, no such $j$ exists and we have a contradiction. The remaining case is $\pi=456123$ and if $j=1$, then $u$ contains $546123$, which contains 4312; if $j=2$, then $u$ contains 465123, which contains 4312; if $j=4$, then $u$ contains 456213, which contains 3421; if $j=5$, then $u$ contains 456132, which contains 3421. 
\end{proof}

\begin{lemma}\label{lem:not2boolean-bad}
Let $w\in\S_n$ be a permutation with a reduced expression $w=s_{i_1}\cdots s_{i_{\ell}}$ where some simple transposition $s_k$ appears at least 3 times, then $w$ contains one of 3421, 4312, 4321, 456123.
\end{lemma}
\begin{proof}
Use induction on $\ell(w)$. If $s_{i_1}\neq s_k$, then $w'=s_{i_2}\cdots s_{i_{\ell}}$ contains $s_k$ at least 3 times so by induction hypothesis, $w'$ contains one of the 3421, 4312, 4321, 456123. By Lemma~\ref{lem:inclength-stillbad}, $w$ contains one of the patterns as well and we are done. Thus, we can assume that $s_{i_1}=s_k$, and similarly $s_{\ell}=s_k$, so that $w=s_k\cdots s_k\cdots s_k$.

As $w$ has a right inversion $s_k$, $w(k)>w(k+1)$. Let $x=w(k+1)$ and $y=w(k)$ with $x<y$. As $w$ has a left inversion $s_k$, we know that $k+1$ appears before $k$ in $w$. Let $w(i)=k+1$ and $w(j)=k$ with $i<j$. If $\{i,j\}=\{k,k+1\}$, then $w(k)=k+1$ and $w(k+1)=w(k)$, so that $ws_k=s_kw$ and $w$ cannot possibly have a reduced expression starting and ending at $s_k$. This case is impossible. We will consider various orderings of $i,j,k,k+1$ and $x,y,k,k+1$ to find patterns in $w$. Write $u=s_kws_k$ so that $\ell(u)=\ell(w)-2$. By Lemma~\ref{lem:si=support}, since a reduced expression of $u$ contains $s_k$, $u$ has an inversion across index $k$. We are going to use this strategy for the following cases.\\

\noindent\textbf{Case 1:} $|\{i,j\}\cup\{k,k+1\}|=3$. We have a few subcases here. 

If $i=k$, then $j>k+1$. As there are at least two values among $\{w(k+1),w(k+2),\ldots,w(n)\}$ that are at most $k$, namely $w(k+1)<w(k)=k+1$ and $w(j)=k$, there must be at least two values $\{w(1),\ldots,w(k)\}$ that are greater than $k$. We already have $w(k)=k+1$ so there exists some $a<k$ such that $w(a)>k$. But $w(a)\neq k+1$ so $w(a)\geq k+2$. As a result, $w$ contains 4312 at indices $a,k,k+1,j$.

If $i=k+1$, then $j>k+1$ and we see that $w(k)>w(k+1)>w(j)$. Then $u(k)=k$, $u(k+1)=w(k)>w(k+1)=k+1$, $u(j)=k+1$. Since $u$ has an inversion across index $k$, we must have some $a\in\{k+1,\ldots,n\}$ such that $u(a)\leq k$. As $u(k)=k$, $u(a)<k$, $a\neq k, k+1, j$. We see that $w(a)=u(a)$, and if $a<j$, $w$ contains 4312 at indices $k<k+1<a<j$ and if $a>j$, $w$ contains 4321 at indices $k<k+1<j<a$. 

If $j=k$, then $i<k$ and $w(i)>w(k)>w(k+1)$. Similar as above, we see that $u(i)=k$, $u(k)=w(k+1)<w(k)=k$, $u(k+1)=k+1$. As $u$ has $s_k$ in its reduced expressions, there exists some $a\in\{1,\ldots,k\}$ such that $u(a)>k$. Thus, $a\neq i,k$ and $u(a)\geq k+2$. Back to $w$, we have $w(a)=u(a)$. So if $a<i$, $w$ contains 4321 at indices $a<i<k<k+1$ and if $a>i$, $w$ contains 3421 at indices $i<a<k<k+1$.

If $j=k+1$, then $i<k$. Both $w(i)$, $w(k)$ are greater than $k$. As $\{w(1),\ldots,w(k)\}\cap\{k+1,\ldots,n\}$ has cardinality at least 2, $\{w(k+1),\ldots,w(n)\}\cap\{1,\ldots,k\}$ has cardinality at least 2. So there exists some $a>k+1$ such that $w(a)<k$. As a result, $w$ contains 3421 at indices $i<k<k+1<a$.

The situation when $|\{x,y\}\cup\{k,k+1\}|=3$ can be deduced from Case 1 by taking inverses. From now on, assume that both $\{i,j\}$ and $\{x,y\}$ are disjoint from $\{k,k+1\}$. Table~\ref{tab:2boolean-case-table} shows how we divide the problem into cases.
\begin{table}[h!]
\centering
\begin{tabular}{c|c|c|c}
& $x<y<k<k{+}1$ & $x<k<k{+}1<y$ & $k<k{+}1<x<y$ \\\hline
$i{<}j{<}k{<}k{+}1$ & Case 2 (4321) & Case 3 (4321/4312) & Case 3 (4321/4312) \\\hline
$i{<}k{<}k{+}1{<}j$ & Case 3 (4321/3421) & Case 5 (...) & Case 4 (4321/4312)\\\hline
$k{<}k{+}1{<}i{<}j$ & Case 3 (4321/3421) & Case 4 (4321/3421) & Case 2 (4321)\\\hline
\end{tabular}
\caption{Cases for the proof of Lemma~\ref{lem:not2boolean-bad}}
\label{tab:2boolean-case-table}
\end{table}

\noindent\textbf{Case 2:} $i{<}j{<}k{<}k{+}1$ and $x{<}y{<}k{<}k{+}1$ or $k{<}k{+}1{<}i{<}j$ and $k{<}k{+}1{<}x{<}y$. In this case, we directly see that $w$ contains 4321 at indices $i,j,k,k+1$ (either $i{<}j{<}k{<}k{+}1$ or $k{<}k{+}1{<}i{<}j$).\\

\noindent\textbf{Case 3:} $i{<}j{<}k{<}k{+}1$ and $y>k+1$. Since $\{w(1),\ldots,w(k)\}\cap\{k+1,\ldots,n\}$ has cardinality at least 2, namely $w(i)=k+1$ and $w(k)=y>k+1$, $\{w(k+1),\ldots,w(n)\}\cap\{1,\ldots,k\}$ must have cardinality at least 2. Say $k<a<b$ and $w(a),w(b)\leq k$. As $w(j)=k$ with $j<k$, we must have $w(a),w(b)<k$. As a result, $w$ contains either 4321 or 4312 at indices $i<j<a<b$. By taking inverses, we are also down with the case where $x{<}y{<}k{<}k{+}1$ and $j>k+1$. \\

\noindent\textbf{Case 4:} $i{<}k{<}k{+}1{<}j$ and $k{<}k{+}1{<}x{<y}$. Since $\{w(1),\ldots,w(k)\}\cap\{k+1,\ldots,n\}$ has cardinality at least 2, namely $w(i)=k+1$ and $w(k)=y>k+1$, $\{w(k+1),\ldots,w(n)\}\cap\{1,\ldots,k\}$ must have cardinality at least 2. Besides $w(j)=k$, we must some $a>k$, $a\neq j$, such that $w(a)<k$. Also $a>k+1$ since $w(k+1)=x>k+1$. As a result, $w$ contains 4321 at indices $k,k+1,j,a$ if $a>j$ and contains 4312 at indices $k,k+1,a,j$ if $a<j$.\\

\noindent\textbf{Case 5:} $i{<}k{<}k{+}1{<}j$ and $x{<}k{<}k{+}1{<}y$. Recall that $u=s_kws_k$. In this case, $u(i)=k$, $u(k)=w(k+1)=x<k$, $u(k+1)=w(k)=y>k+1$, $u(j)=k+1$. Since a reduced expression of $u$ uses $s_k$, we cannot possibly have $\{u(1),\ldots,u(k)\}=\{1,\ldots,k\}$. There exists $a<k$ such that $u(a)>k$ and $b>k$ such that $u(b)<k$. Since $u(j)=k+1$, $u(a)>k+1$, and also $a\neq i, k$. Similarly, $u(b)<k$ and $b\neq k+1, j$. This also tells us $u(a)=w(a)$, $u(b)=w(b)$. If $a<i$, then $w$ contains 4312 at indices $a,i,k+1,j$ and if $w(a)>y$, then $w$ contains 4312 at indices $a,k,k+1,j$. Similarly, if $b>j$, then $w$ contains 3421 at indices $i,k,j,b$ and if $w(b)<x$, then $w$ contains 3421 at indices $i,k,k+1,b$.  The final remaining case is that $i<a<k$, $k+1<w(a)<y$, $k+1<b<j$, $x<w(b)<k$, where $w$ contains 456123 at indices $i,a,k,k+1,b,j$.
\end{proof}

Now Theorem~\ref{thm:2boolean} follows from Lemma~\ref{lem:bad-not2boolean} and Lemma~\ref{lem:not2boolean-bad}.

\subsection{Enumeration of 2-boolean permutations}\label{sub:enumeration2boolean}
Throughout this section, let $f(n)$ denote the number of 2-boolean permutations in $\S_n$. We adopt the convention that $f(0)=1$. We have that $f(1)=1$, $f(2)=2$, $f(3)=6$, $f(4)=21$, $f(5)=78$, and so on, which appears as sequence A124292 in OEIS \cite{sloane2008line}.
\begin{theorem}\label{thm:enumeration}
Let $f(n)$ be the number of 2-boolean permutations in $\S_n$. Then
$$\sum_{n\geq0}f(n)q^n=\frac{1-5q+5q^2}{1-6q+9q^2-3q^3}.$$
\end{theorem}
In this section, we think of 2-boolean permutations as permutations that avoid 3421, 4312, 4321 and 456123 (Theorem~\ref{thm:2boolean}). Let's first look at what a typical 2-boolean permutation looks like. Let $w$ be 2-boolean. If $w(1)=1$, then $w$ restricted to indices $2,3,\ldots,n$ is just a 2-boolean permutation in $\S_{n-1}$ (and it is easy to see that this can in fact be any $2$-boolean permutation in $\S_{n-1}$). If $w(1)\neq1$, we define the following sets:
\begin{align*}
C(w)=&\{(i,w(i))\:|\: 1<i<w^{-1}(1),1<w(i)<w(1)\},\\
A(w)=&\{(i,w(i)\:|\: 1<i<w^{-1}(1),w(i)>w(1))\},\\
B(w)=&\{(i,w(i))\:|\: i>w^{-1}(1),1<w(i)<w(1)\}.
\end{align*}
Write $a(w)=|A(w)|$, $b(w)=|B(w)|$ and $c(w)=|C(w)|$ for cardinality. Note that all these quantities are only defined for those $w$ such that $w(1)\neq1$. 
See Figure~\ref{fig:2boolean-structure} for a visual description of these regions.
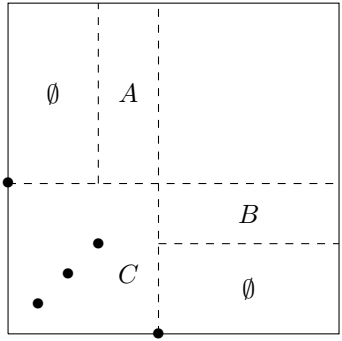
\begin{figure}[h!]
\centering
\begin{tikzpicture}[scale=0.8]
\draw(0,0)--(5.5,0)--(5.5,5.5)--(0,5.5)--(0,0);
\node at (0,2.5) {$\bullet$};
\node at (2.5,0) {$\bullet$};
\draw[dashed](2.5,0)--(2.5,5.5);
\draw[dashed](0,2.5)--(5.5,2.5);
\node at (0.5,0.5) {$\bullet$};
\node at (1,1) {$\bullet$};
\node at (1.5,1.5) {$\bullet$};
\draw[dashed](2.5,1.5)--(5.5,1.5);
\draw[dashed](1.5,2.5)--(1.5,5.5);
\node at (2,1) {$C$};
\node at (2,4) {$A$};
\node at (4,2) {$B$};
\node at (0.75,4) {$\emptyset$};
\node at (4,0.75) {$\emptyset$};
\end{tikzpicture}
\caption{Structure of a 2-boolean permutation}
\label{fig:2boolean-structure}
\end{figure}
Since $w$ avoids 4321, we see that entries in $C(w)$ must be increasing. Let $C(w)=\{(i_1,w(i_1)),\ldots,(i_c,w(i_c))\}$ with $i_1<\cdots<i_c$ and $w(i_1)<\cdots<w(i_1)$. As $w$ avoids 3421, the region $\{(i,w(i))\:|\: 1\leq i\leq i_c,w(i)>w(1)\}$ must be empty. Similarly as $w$ avoids 4312, the region $\{(i,w(i))\:|\: i>w^{-1}(1),1<w(i)<w(i_c)\}$ is empty. These empty sets are indicated in Figure~\ref{fig:2boolean-structure}. Consequently, we know that $C=\{(2,2),(3,3),\ldots,(c+1,c+1)\}$ where $c=c(w)$. Moreover, it is impossible for $a(w)\geq2$ and $b(w)\geq2$ to happen simultaneously. Otherwise, say $A(w)$ contains $(x_1,w(x_1))$ and $(x_2,w(x_2))$ while $B(w)$ contains $(y_1,w(y_1))$ and $(y_2,w(y_2))$ with $x_1<x_2$ and $y_1<y_2$. If $w(x_1)>w(x_2)$, then $w$ contains 4312 at indices $x_1,x_2,w^{-1},y_1$ and similarly if $w(y_1)>w(y_2)$, $w$ contains 3421 at indices $1,x_1,y_1,y_2$; and if finally $w(x_1)<w(x_2)$ and $w(y_1)<w(y_2)$, then $w$ contains 456123 at indices $1,x_1,x_2,w^{-1}(1),y_1,y_2$. As a result, either $a(w)\leq1$ or $b(w)\leq1$ for a 2-boolean permutation $w$.

As an important piece of notation, we use
$f^{a,b}_c(n)$ to denote the number of 2-boolean permutations $w$ in $\S_n$ such that $a(w)=a$, $b(w)=b$ and $c(w)=c$. Note that $f^{a,b}_c(n)=f^{b,a}_c(n)$ by the symmetry of taking inverses. We will also omit some superscripts or subscripts to mean we require less conditions. For example, $f^a(n)$ is the number of 2-boolean permutations $w$ in $\S_n$ with $a(w)=a$. 

The following lemma is the key to our recurrence.
\begin{lemma}\label{lem:count-eachterm}
We have the following identities for $n\geq4$:
\begin{align}
f^{0}(n)=&\sum_{1\leq k\leq n-1}f(k),\\
f^1(n)=&f(n-1)-f(n-2),\\
f^{0,0}(n)=&\sum_{0\leq k\leq n-2}f(k),\\
f^{0,1}(n)=&\sum_{1\leq k\leq n-2}f(k),\\
f^{1,1}(n)=&f(n-2)-1.
\end{align}
\end{lemma}
\begin{proof} We will start by proving (1). Note that partitioning 2-boolean permutations according to the value of $c$, we obtain
\[f^0(n)=\sum_{c=0}^{n-2}f^0_c(n).\]

We will now show that $f^0_c(n)=f(n-1-c)$. Consider a 2-boolean permutation $w\in \S_n$ with $c(w)=c$ and $a(w)=0$. Let $w'$ be the restriction of $w$ to the indices $1,c+3,c+4,\ldots, n$. Note that $w'\in \S_{n-1-c}$ is a $2$-boolean permutation. Furthermore, this map $w\to w'$ takes different 2-boolean $w$ with $c(w)=c$ to different $2$-boolean $w'\in \S_{n-1-c}$. Also note that for any $2$-boolean $w'\in \S_{n-1-c}$, if we construct a permutation $w\in \S_n$ by letting $w(2)=2, w(3)=3, \ldots, w(c+1)=c+1$, $w(c+2)=1$, and we let the restriction of $w$ to the indices $1, c+3, c+4, \ldots, n$ be equal to $w'$, then $w$ is also $2$-boolean with $a(w)=0$ and $c(w)=c$. To see this, observe that for $w$ to contain $3421$, $4312$, $4321$, or $456123$, some image $\leq c+1$ would have to be $3$, $4$, or $5$ in the pattern, which is impossible. This map is also injective. Hence, we have a bijection showing that $f^0_c(n)=f(n-1-c)$. This lets us rewrite the above sum as
\[f^0(n)=\sum_{c=0}^{n-2}f^0_c(n)=\sum_{c=0}^{n-2}f(n-1-c)=\sum_{k=1}^{n-1} f(k),\]

The proof of (3) is completely analogous to the proof of (1), with the only difference being that we instead consider the restriction of $w$ to the indices $c+3, c+4, \ldots, n$, and note that this can be any 2-boolean $w'\in \S_{n-2-c}$, whereas $w(1)=c+2$, $w(2)=2, w(3)=3, \ldots, w(c+1)=c+1, w(c+2)=1$. The fourth identity is also analogous, with the restriction being to the same indices $c+3, c+4, \ldots, n$, and the fixed values being $w(1)=c+3$, $w(2)=2, \ldots, w(c+1)=c+1$.

As for (2), consider a $2$-boolean $w\in \S_n$ with $a(w)=1$ and $c(w)=c$. Then $w(2)=2, w(3)=3, \ldots, w(c+1)=c+1$, and $w(c+3)=1$. The restriction of $w$ to the rest of the indices $1,c+2, c+4, c+5, \ldots, n$ is a 2-boolean permutation $w'\in \S_{n-c-1}$ with $w'(1)<w'(2)$. Furthermore, any such permutation $w'$ can be inserted to these indices while giving a $2$-boolean $w$. These maps are inverses of each other, so it suffices to count the number of such permutations. For this, it suffices to count the size of the complement, i.e. the number of 2-boolean permutations $u\in \S_{n-c-1}$ with $u(1)>u(2)$. This is equivalent to $u(1)>1$ and $c(u)\geq 1$ or $c(u)=0, a(u)=0$.

For the first case, i.e. that $u(1)>1$ and $c(u)\geq 1$, we can count the number of such $u\in \S_{n-c-1}$ in the following way. Note that $u(2)=2$ (since $c(u)\geq 1$, and the restriction of $u$ to the rest of the indices $1,3,4,\ldots, n-c-1$ is a $2$-boolean permutation $u'\in \S_{n-c-2}$ such that $u(1)>1$. Furthermore, when we insert any such permutation to these indices, no bad pattern is created that involves the index $2$. These maps are clearly inverses of each other, so we have a bijection. The number of $2$-boolean $u'\in \S_{n-c-2}$ with $u(1)>1$ is $f(n-c-2)-f(n-c-3)$. By our bijection, this is also the number of 2-boolean $u\in \S_{n-c-1}$ with $u(1)>1$ and $c(u)\geq 1$.

For the second case, i.e. that $u(1)>1, c(u)=0,$ and $a(u)=0$, the number of such $u\in \S_{n-c-1}$ is $f^0_0(n-c-1)$, which is equal to $f(n-c-2)$ as argued before.

Putting everything together and summing over $c$, we get that the number of $2$-boolean $w$ with $a(w)=1$ is 
\[f^1(n)=\sum_{c=0}^{n-3} f(n-c-1)-\left(\left(f(n-c-2)-f(n-c-3)\right)+f(n-c-2)\right)\]\[=\sum_{c=0}^{n-3}f(n-c-1)-2f(n-c-2)+f(n-c-3).\]

This sum telescopes, and we are left with the desired
\[f^1(n)=f(n-1)-f(n-2)-f(0)+f(1)=f(n-1)-f(n-2).\]

It remains to show (5). Consider a permutation $w$ with $a(w)=b(w)=1$, and $c(w)=c$. Then $w(1)=c+3$, $w(2)=2, w(3)=3, \ldots w(c+1)=c+1$, and $w(c+3)=1$. Let $w'$ be the restriction of $w$ to the rest of the indices $c+2,c+4,c+5,\ldots, n$. Then $w'$ is a $2$-boolean permutation in $S_{n-c-2}$ with $w'(1)\neq 1$. Conversely, when we insert any such permutation to these indices, no bad pattern is created, and the result is $w$ with $a(w)=b(w)=1, c(w)=c$. This gives a bijection showing that $f^{1,1}_c(n)=f(n-c-2)-f(n-c-3)$. We sum over $c$:
\[f^{1,1}(n)=\sum_{c=0}^{c=n-4}f^{1,1}_c(n)=\sum_{c=0}^{c=n-4}f(n-c-2)-f(n-c-3).\]
This sum telescopes, and we are left with the desired identity
\[f^{1,1}(n)=f(n-2)-f(1)=f(n-2)-1.\]

\end{proof}

We are now ready to finish the proof of Theorem~\ref{thm:enumeration}.
\begin{proof}[Proof of Theorem~\ref{thm:enumeration}]
Let $n\geq4$. Recall that for a 2-boolean permutation $w\in\S_n$ with $w(1)\neq1$, we have either $a(w)\leq1$ or $b(w)\leq1$. By the symmetry between $A(w)$ and $B(w)$, and by simple inclusion-exclusion, we obtain
$$f(n)=f(n-1)+2f^0(n)+2f^1(n)-f^{0,0}(n)-2f^{0,1}(n)-f^{1,1}(n)$$
where the term $f(n-1)$ accounts for those permutations $w$ with $w(1)=1$ while $f^0(n)$ accounts for those with $a(w)=0$ and another $f^0(n)$ corresponds to those with $b(w)=0$ and so on. For simplicity of notation, write $S=\sum_{k=1}^{n-3}f(k)$. By Lemma~\ref{lem:count-eachterm}, we continue the computation
\begin{align*}
f(n)=&f(n-1)+2(f(n-1)+f(n-2)+S)+2(f(n-1)-f(n-2))\\
&-(f(n-2)+S+1)-2(f(n-2)+S)-(f(n-2)-1)\\
=&5f(n-1)-4f(n-2)-S+1.
\end{align*}
For $n\geq5$, we write the above equation using $n-1$ to get
$$f(n-1)=5f(n-2)-4f(n-3)-\sum_{k=1}^{n-4}f(k)+1.$$
Subtract from the above computation, we obtain a linear recurrence
$$f(n)-6f(n-1)+9f(n-2)-3f(n-3)=0$$
for $n\geq5$. Together with the initial terms $f(0)=f(1)=1$, $f(2)=2$, $f(3)=6$ and $f(4)=21$, we obtain the desired generating function.
\end{proof}

\section*{Acknowledgements}
This research is conducted during UROP (Undergraduate Research Opportunity Programs) in Spring 2020 at MIT. K.H. is supported by the Dae Y. and Young J. Lee UROP Fund. We thank Alex Postnikov for many illuminating ideas and Bridget Tenner for helpful conversations.

\bibliographystyle{plain}

\begin{thebibliography}{10}

\bibitem{billey2005smoothness}
Sara Billey and Alexander Postnikov.
\newblock Smoothness of {S}chubert varieties via patterns in root subsystems.
\newblock {\em Adv. in Appl. Math.}, 34(3):447--466, 2005.

\bibitem{bjornerbrenti}
Anders Bj\"{o}rner and Francesco Brenti.
\newblock {\em Combinatorics of {C}oxeter groups}, volume 231 of {\em Graduate
  Texts in Mathematics}.
\newblock Springer, New York, 2005.

\bibitem{gaetz2019splittings}
Christian Gaetz and Yibo Gao.
\newblock Separable elements and splittings of weyl groups.
\newblock {\em arXiv preprint arXiv:1911.11172}, 2019.

\bibitem{gaetz2019separable}
Christian Gaetz and Yibo Gao.
\newblock Separable elements in {W}eyl groups.
\newblock {\em Adv. in Appl. Math.}, 113:101974, 23, 2020.

\bibitem{hohlweg2016On}
Christophe Hohlweg and Jean-Philippe Labb\'{e}.
\newblock On inversion sets and the weak order in {C}oxeter groups.
\newblock {\em European J. Combin.}, 55:1--19, 2016.

\bibitem{humphreys1978introduction}
James~E. Humphreys.
\newblock {\em Introduction to {L}ie algebras and representation theory},
  volume~9 of {\em Graduate Texts in Mathematics}.
\newblock Springer-Verlag, New York-Berlin, 1978.
\newblock Second printing, revised.

\bibitem{Lakshmibai-Sandhya}
V.~Lakshmibai and B.~Sandhya.
\newblock Criterion for smoothness of {S}chubert varieties in {${\rm
  Sl}(n)/B$}.
\newblock {\em Proc. Indian Acad. Sci. Math. Sci.}, 100(1):45--52, 1990.

\bibitem{sloane2008line}
Neil~JA Sloane et~al.
\newblock The on-line encyclopedia of integer sequences, 2020.

\bibitem{tenner2007pattern}
Bridget~Eileen Tenner.
\newblock Pattern avoidance and the {B}ruhat order.
\newblock {\em J. Combin. Theory Ser. A}, 114(5):888--905, 2007.

\bibitem{tenner2020interval}
Bridget~Eileen Tenner.
\newblock Interval structures in the bruhat and weak orders.
\newblock {\em arXiv preprint arXiv:2001.05011}, 2020.

\bibitem{woo2018hultman}
Alexander Woo.
\newblock Hultman elements for the hyperoctahedral groups.
\newblock {\em Electron. J. Combin.}, 25(2):Paper No. 2.41, 25, 2018.

\end{thebibliography}

\end{document}